\documentclass[11pt,twoside,english,a4paper]{article}

\usepackage{babel,amssymb}
\usepackage{graphicx}
\usepackage{amsmath}
\numberwithin{equation}{section}

\usepackage{algorithm}
\usepackage{algpseudocode}
\usepackage{bm}

\usepackage{tikz}
\usepackage{pgfplots}
\pgfplotsset{compat=1.18}
\usetikzlibrary{plotmarks,arrows.meta}
\usepgfplotslibrary{patchplots}

\usepackage{hyperref}
\usepackage{orcidlink}

\usepackage[markup=final]{changes}
\definechangesauthor[name={R2}, color=red]{R2}
\setauthormarkuptext{id}

\newtheorem{proposition}{Proposition}[section]

\newtheorem{lemma}[proposition]{Lemma}
\newtheorem{theorem}[proposition]{Theorem}
\newtheorem{notita}[proposition]{Remark}
\newenvironment{remark}{\begin{notita}\rm}{\hfill$\Box$\\[0.5ex]\end{notita}}
\newenvironment{proof}{{\it Proof}. }{\hfill$\Box$\\[0.5ex]}

\DeclareMathOperator*{\argmin}{arg\,min}

\title{Fast and stable computation of highly oscillatory and/or exponentially decaying integrals using a Clenshaw-Curtis product-integration rule}
\author{V\'{\i}ctor Dom\'{\i}nguez\thanks{Dep. Estadística, Informática y Matemáticas. Campus de Tudela 31500 - Tudela, Spain. INAMAT2 Institute for Advanced Materials and Mathematics.
E-mail: victor.dominguez@unavarra.es.
ORCID: \orcidlink{0000-0002-6095-619X}}}

\date{\today}

\begin{document}

\maketitle 

\begin{abstract}
We propose, analyze, and implement a quadrature method for evaluating integrals of the form $\int_0^2 f(s)\exp(zs)\, {\rm d}s$, where $z$ is a complex number with a possibly large negative real part. The integrand may exhibit exponential decay, highly oscillatory behavior, or both simultaneously, making standard quadrature rules computationally expensive. Our approach is based on a Clenshaw-Curtis product-integration rule: the smooth part of the integrand is interpolated using a polynomial at Chebyshev nodes, and the resulting integral is computed exactly. We analyze the convergence of the method with respect to both the number of nodes and the parameter $z$. Additionally, we provide a stable and efficient implementation whose computational cost is essentially independent of $z$ and  {scales linearly with the number of Chebyshev nodes}. Notably, our approach avoids the use of special functions, enhancing its numerical robustness.
\end{abstract}



\section{Introduction}
In this work, we consider the numerical approximation of the following model integral, expressed in the form we will use throughout this work: for $ \Re z \leq {\mu_0} $, where $ {\mu_0} $ is either negative or has a moderate value (to prevent the integral from becoming excessively large),
\begin{equation}\label{eq:theintegral:0}
 \int_0^{2} f(s) \exp(s z) \, {\rm d}s.
\end{equation}
When $\Re z=0$, these integrals are often referred to in the literature as Fourier-type oscillatory integrals (cf. \cite{Is:2003}) and have received considerable attention in recent years in the broader field of oscillatory integrals involving the general exponential term $\exp({\rm i}\eta g(s))$. We briefly mention several methods for such oscillatory integrals: The Steepest Descent methods \cite{Huybrechs2006, DeHu:2009}, which rewrite the integral as a non-oscillatory contour integral in the complex plane with exponential decay, the Levin methods \cite{MR645668}, which reformulate the problem by reducing it to the solution of a suitable ordinary differential equation, and the Filon-type methods \cite{Is:2003, MR4244341, Xi:2007}, similar to those considered in this paper, based on replacing the function $f$ in the integral by a suitable polynomial and computing this new integral analytically. We refer the reader to the textbook \cite{MR3743075} and references therein for a comprehensive introduction to this topic.

In the case of smooth functions $f$, the idea of using the product integration rule in \eqref{eq:theintegral:0}
\begin{equation}\label{eq:themethod}
\int_0^{2} p_L(s) \exp(s z)\, {\rm d}s,
\end{equation}
where $p_L$ is an interpolating polynomial at $L+1$ well-distributed nodes, making it a natural approach to consider. Given their excellent properties, Chebyshev nodes—defined as the set of extrema of $T_{\ell+1}'(1+\cdot)$ together with the endpoints $\{0,2\}$—are a natural choice. Indeed, again when $\Re z=0$—therefore, the integral becomes oscillatory—the so-called Filon-Clenshaw-Curtis quadrature rules have gained popularity due to their excellent performance for small, moderate, and large values of $z$. We refer to the pioneering work of Piessens et al. \cite{PiPo:1971}, which led to the implementation of a closely related rule in the Fortran library QUADPACK \cite{QUADPACK}. This method can be framed within the broader Chebyshev-based approach to numerical computations, which has remained popular since then. See, for instance, the survey \cite{Pi:2000} or \cite{Tr:2008} for quadrature rules, as well as the extensive use of Chebyshev polynomials in numerical computations within the Chebfun package \cite{Driscoll2014}.

The convergence {analysis} of this rule, along with its fast and stable implementation for purely oscillatory integrals, was previously carried out in \cite{DoGrSm:2010}. There, we showed that this rule exhibits a {\em dual convergence behavior}: We have convergence in $L$, the number of nodes of the rule, as a consequence of the convergence of the interpolant and a Filon-type convergence as $|z| \to \infty$ since the error decays as $\mathcal{O}(|z|^{-2})$ for fixed $L$ for smooth enough functions $f$. The convergence rate with respect to $L$ depends on the regularity of $f$; for smooth functions $f$, it can be superalgebraic or even exponential for analytic integrands (cf. \cite{XiChWa:2010, Xiang:2016} or \cite[Ch.~8]{Trefethen2020}). This convergence in $z$ follows, instead, from an {integration-by-parts argument}, which exploits the fact that the endpoints of the integration interval, $0$ and $2$, serve as nodes in the interpolation problem that defines $p_L$. Further improvement in convergence with respect to $z$ can be achieved by using a modification of this technique involving Hermite interpolation, yielding an $\mathcal{O}(|z|^{-r-1})$ convergence rate if the first $r$ derivatives are interpolated at the endpoints. This modification, which we do not explore in this work, can be found in related contexts (see, for instance, \cite{GaIs:2017, Is:2003,IseNor:2005}).

Integrals that combine both oscillatory and decaying behavior, as in \eqref{eq:theintegral:0} with $z \in \mathbb{C}_-$ (i.e., a non-positive real part), have received comparatively less attention. One possible reason is that when $ \Re z \leq 0 $, the integrand can often be split into two parts: a subinterval where it is exponentially small and can safely be approximated by zero, and another where standard or oscillatory quadrature rules can be applied. This strategy is somewhat similar to those used in the aforementioned Steepest Descent methods, which reduce oscillatory integrals to exponentially decaying contour integrals in the complex plane. However, this approach is not always practical.
 For instance, when the integral \eqref{eq:theintegral} must be computed for a very wide range of $ z $ values, the function $ f $ may need to be evaluated at many different points depending on how $ z $ influences the splitting. Thus, designing robust quadrature rules that remain stable for varying values of $z$ is a worthwhile endeavor.

At first glance, the convergence in $L$ and $z$ of the rule may appear to be an extension of previous work on oscillatory cases, but key subtleties distinguish it from previous approaches. For instance, previous convergence analyzes relied on a change of variables (a cosine transformation) that converted the integral into a $\pi$-periodic form. Since we no longer assume $\Re z = 0$, we encounter substantial differences that require a different convergence {analysis}. Indeed, we are forced to perform the {analysis} using the original variables instead of the cosine change of variables applied in \cite{DoGrSm:2010}.

In analogy with the oscillatory case, the interpolating polynomial can be written in terms of Chebyshev polynomials using FFT techniques, which reduces the applicability of the rule to computing coefficients
\begin{equation}\label{eq:Tn1}
 \int_{0}^2 T_n(s-1)\exp(s z) \, {\rm d}s,
\end{equation}
where $T_n$ denotes the Chebyshev polynomial of degree $n$. Such coefficients satisfy a recurrence relation very similar to the periodic case; however, unlike that case, it proves significantly less stable, rendering the approach we proposed in \cite{DoGrSm:2010} valid only for $|L| \lesssim 2|z|^{1/2}$.
This should be compared to the stability condition $|L|\leq |z|$ when $\Re z=0$, which was proven analytically and verified numerically in the pure oscillatory case. This situation represents a not uncommon mid-range scenario when aiming for a sufficiently accurate approximation of the integral, which requires $L$ to be of moderate size and $|z| \lesssim L^2$. For this reason, among others, the algorithms developed for computing \eqref{eq:Tn1} in the oscillatory case are not directly applicable and require a complete redesign. This is precisely one of the key contributions: our new algorithms can perform these calculations for any $z$ with a very similar cost.

We conclude this introduction by considering possible extensions of the present work. First, we could address the case where $f$ presents singularities. There are two immediate approaches to consider: either factorizing the singularity and incorporating it into the weight functions alongside the complex exponential, or using graded meshes towards the singularities. The first approach requires a precise identification of the singularity, modifying the weights in \eqref{eq:Tn1} accordingly, and developing new algorithms for their computation. Examples of such techniques, applied specifically in the pure oscillatory case, can be found in \cite{Dom:2014,Pi:2000,PIESSENS1992159} for the first approach and in \cite{DoGrKi:2013} for the second one.

On the other hand, we can consider more complex {\em oscillatory-decaying} behavior, i.e., cases where the exponential term takes the alternative form $\exp(g(s)z)$. Let us emphasize that one of the key ideas in \cite{DoGrKi:2013} was to reduce such cases to the setting considered in this paper by performing the change of variables $\tilde{s} = g(s)$.
This transformation can introduce a new singularity in the integrand if $g(s)$ has stationary points, that is, points $s_1$ where the derivative of $g(s)$ vanishes, thus making the integral fall within the framework of singular integrals discussed above. We leave these two extensions for future work.

This paper is structured as follows: In Section 2, we introduce the integral and the proposed product quadrature rule for its approximation, discussing practical implementation aspects once the coefficients are computed. Section 3 provides convergence estimates in terms of $ L $, the number of nodes, and $ |z| $. In Section 4, we focus on the computation of weights that make the quadrature rule usable, providing robust algorithms for this purpose. Finally, in Section 5, we present extensive numerical experiments that demonstrate the method's efficiency and validate our theoretical results. Additionally, we illustrate how these quadrature rules can be incorporated as an intermediate step into more advanced methods, such as Laplace-based numerical solvers for fractional-order evolution problems.

%
% \replaced[id=R2]%
% {A MATLAB implementation of the algorithm is available on GitHub \cite{CleCurRules}, which also includes codes in the {\em training area} folder that, after minor modifications, can be used to replicate some of the numerical experiments in the last section. In addition, a Python implementation is provided in \cite{CleCurRulesPython} with the aim of making the algorithm more accessible and easier to integrate into a broader range of scientific workflows.}%
% {A Matlab implementation of the algorithm is available on  GitHub \cite{CleCurRules}, which also includes codes in the {\em training area} folder that, after minor modifications, can be used to replicate some of the numerical experiments in the last section.}%

{A Matlab implementation of the algorithm is available on  GitHub \cite{CleCurRules}, which also includes codes in the {\em training area} folder that, after minor modifications, can be used to replicate some of the numerical experiments in the last section.}%

\section{Oscillatory and/or exponentially decaying integrals and product Clenshaw-Curtis rules}\label{sec:QR}

In this work, we define
\begin{equation}
\label{eq:theintegral} I_{z}(f):=\int_0^2 \exp(s z)f(s)\,{\rm d}s. \end{equation}
where
\begin{equation}
\label{eq:defCmu} z\in{\mathbb C}_{\mu_0}:= \left\{z\in\mathbb{C} \ :\ \sigma:=\Re z\leq \mu_0\right\}
\end{equation}
where $\mu_0$ is a negative or moderately small positive real number (to prevent numerical blow-up of the integral). Therefore, depending on $z$, the integrand may exhibit oscillatory behavior, exponential decay, or a combination of both.

{We also observe that, by integration by parts,
\begin{equation}
|I_{z}(f)| \le \frac{1}{|z|}\left(|f(0)| + |f(2)| e^{2\sigma}\right)
  + \frac{1}{|z|}\left|\int_{0}^{2} f'(s)e^{ zs}\,{\rm d}s\right|,
\end{equation}
which provides the asymptotic behavior of the integral we wish to approximate. Furthermore, by iterating the same argument and invoking the Sobolev embedding theorem together with a Poincaré-type inequality, we easily see that whenever $f^{(m)} \in L^1(0,2)$ and its derivatives up to order $m-1$ vanish at $0$, then for all $|z| > \zeta_{0} > 0$ there exists a constant $c_{m,\zeta_{0}} > 0$, independent of $f$, such that
\begin{equation}
|I_{z}(f)| \le c_{m,\zeta_{0}}\left(|z|^{-1}e^{2\sigma} + |z|^{-m}\right)\int_{0}^2 |f^{(m)}(s)|\,{\rm d}s.
\end{equation}}

The proposed approximation of $ I_{z}(f) $ relies on the (modified or product) Clenshaw-Curtis method. Thus, for an integer $L \ge 1$, we construct the interpolating polynomial using the (shifted) Chebyshev nodes:
\begin{equation}
\label{eq:theIntegral}
\mathbb{P}_{L} \ni Q_L f, \quad
(Q_L f)\Big(1+\cos \frac{\ell \pi}{L} \Big)
= f\Big(1+\cos \frac{\ell \pi}{L} \Big),
\quad \ell = 0, \ldots, L,
\end{equation}
and use it to obtain the numerical approximation given by
\begin{equation}
\label{eq:therule}
 I_{L,z}(f):=\int_{0}^2 (Q_L f)(s)\exp(s z)\,{\rm d}s\approx I_{z}(f). 
\end{equation}
The interpolant $Q_L f$ can be written as
\[
 Q_Lf{(s)}= \sum_{\ell=0}^L {}'' \alpha_{\ell,L}(f) T_\ell(s-1)\,
\]
where, in the sum above,
\[
T_\ell(x) = \cos(\ell\arccos(x)),
\]
is the Chebyshev polynomial of the first kind and degree $\ell$, and
$\sum''$ means that the first and last terms are halved. Furthermore, the vector of
coefficients $(\alpha_{\ell,L})_{\ell=0}^L$ can be computed by
\begin{equation}\label{eq:alphal}
 \alpha_{\ell,L}(f) = \frac{2}{L} \sum_{j=0}^L {}''
\cos\Big(\frac{j \ell \pi }{L}\Big) f\Big( 1+\cos(j\pi/L)\Big),
\qquad \ell=0,\ldots, L.
\end{equation}
These calculations correspond to the Discrete Cosine Transform of type I (DCT-I) (see, for instance,
\cite[\S 4.7.25]{DaAA:2008}), which can be efficiently computed via FFT techniques in $ \mathcal{O}(L\log L) $ operations.

Then, summarizing,
\begin{subequations}\label{eq:thequadformula}
\begin{equation}\label{eq:thquadformula:a}
 I_{L,z}(f)=\sum_{\ell=0}^L {}'' \alpha_{\ell,L}(f) \omega_\ell (z),
\end{equation}
where
\begin{equation}\label{eq:thequadformula:b}
 \omega_\ell(z):=\int_{0}^2 T_\ell(s-1) \exp(sz)\,{\rm d}s.
\end{equation}
\end{subequations}
We point out that, although these coefficients can theoretically be evaluated analytically, the resulting expressions are highly complex and numerically unstable. A stable algorithm for this problem is postponed and analyzed in Section 4.

\begin{remark}
 We have formulated the model integral on the interval $[0,2]$ for two reasons. First, the presence of the complex exponential forces us either to work with intervals starting at zero or, on arbitrary intervals $[a,b]$ to write the exponential term as $\exp((s+a)z)$. This results in a less convenient notation.

 Second, the interval $[0,2]$ preserves the characteristic length~$2$ of the Chebyshev polynomials $T_\ell$, a feature that will be used repeatedly, particularly in the next section. One could certainly work on $[0,1]$ by using the standard rescaling $T_\ell(2s-1)$ of the Chebyshev polynomials. However, this modification affects several identities involving these polynomials—especially those related to derivatives and changes of variables. Although these adjustments are straightforward, we prefer to keep the identities as close as possible to those satisfied by the classical Chebyshev polynomials, which motivates our choice.
In any case, integrals defined on a general interval $[a,b]$ can, of course, be rewritten via a simple affine transformation into the form considered in this work.
\end{remark}

\section{Error for the product Clenshaw-Curtis rule}

This section aims to derive convergence estimates for the quadrature rule in terms of both the number of nodes, $L$, and the complex parameter, $z$. Clearly, a first estimate can be straightforwardly derived from
\begin{equation}
\label{eq:intr:01}
| I_{z}(f)- {I_{L,z}}(f)|\leq C(\sigma)\|e_L\|_{L^1(0,2)},\quad \text{with } C(\sigma):=\begin{cases}
                           \big|\frac{e^{2\sigma}-1}{\sigma}\big |, &\sigma\neq 0,\\
                               2,& \sigma=0,
                                                           \end{cases}
\end{equation}
where
\begin{equation}\label{eq:def:eL}
 e_L:= f-Q_{L}f
\end{equation}
is the interpolation error. Hence, using Cauchy-Schwarz inequalities, convergence in $L^2$ weighted norms for the interpolation error suffices to derive error estimates for the quadrature rule. We then establish key convergence results for the interpolant and its first two derivatives in appropriate $L^2$-weighted and $L^\infty$ norms. While some of these results—particularly those related to Chebyshev polynomial properties—are well-documented in the literature, we include the proofs here for completeness.
Let
\begin{equation}\label{eq:21}
 \varrho_0(s):=(s(2-s))^{-1/2},
\end{equation}
and define the associated weighted $L^2$ Hilbert space
\[
L^2_{\varrho_0} := \left\{ f \in L^1(0,2) \; : \; \|f\|_{\varrho_0} < \infty \right\}, \quad
  \|f\|_{\varrho_0}^2 := \int_0^2 |f(s)|^2\varrho_0(s)\,{\rm d}s,
\]
endowed with the inner product naturally defined by the corresponding weighted integral.

An orthonormal basis for this space is given by (see, for instance, \cite[\S 1.5]{Riv:1990}, \cite[Ch. 22]{AbrSte:1964}, or \cite[\S 18.3]{NIST:DLMF})
\[
\Big\{ \frac{1}{\sqrt{\pi}}, \sqrt{\frac{2}{\pi}} \widetilde{T}_1(s), \sqrt{\frac{2}{\pi}} \widetilde{T}_2(s), \ldots \Big\},
\]
where $ \widetilde{T}_n(s) $ is the shifted Chebyshev polynomial of the first kind and degree $n$:
\[
\widetilde{T}_n(s) := \cos\left( n \arccos(s - 1) \right),\quad s \in[0,2].
\]
Thus, any function $f\in L^2_{\varrho_0}$ in this space can be expanded as
\[
f = \frac{1}{2} \widehat{f}(0) + \sum_{n=1}^\infty \widehat{f}(n) \widetilde{T}_n,
\]
where
\[
\widehat{f}(n) := \frac{2}{\pi} \int_0^2 f(s) \widetilde{T}_n(s) \varrho_0(s) \, \mathrm{d}s, \quad n = 0, 1, \ldots
\]
with convergence in $ L^2_{\varrho_0} $. Additionally, the Bessel identity holds
\[
\|f\|_{\varrho_0}^2 = \frac{\pi}{2} \Big[ \frac{1}{2} |\widehat{f}(0)|^2 + \sum_{n=1}^\infty |\widehat{f}(n)|^2 \Big].
\]
Our goal is to extend this relation to the closely related $L^2$-weighted integral and Fourier-coefficient weighted $\ell^2$-based norms:
\begin{equation}
\begin{aligned}
  \|f\|_{\varrho_{m}}^2\ &:=\ \int_0^2 |f(s)|^2\varrho_{m}(s)\,{\rm d}s,
 \quad \text{where } \varrho_{m} (s)\ :=\ (s(2-s))^{m-1/2}, \\
 \|f\|_{m}^2\ &:=\ \frac{\pi}{2} \Big[ |\widehat{f}(0)|^2 + \sum_{n=1}^\infty |n|^{2m} |\widehat{f}(n)|^2 \Big].
\end{aligned}
\end{equation}
The following lemma establishes the first step towards the required relation:
\begin{lemma}
For $n \ge m+1 \ge 1$, and for $f$ sufficiently smooth, the following identity holds:
\[
  \int_{0}^2 f'(s) \widetilde{T}_{n}^{(m+1)}(s) \varrho_{m+1}(s) \,{\rm d}s =  (n^2 - m^2)
  \int_{0}^2 f(s) \widetilde{T}_{n}^{(m)}(s) \varrho_{m}(s) \,{\rm d}s.
\]
\end{lemma}
\begin{proof}
The Chebyshev polynomial $T_n$ satisfies the differential equation
\[
 (1 - x^2) y''(x) - x y'(x) + n^2 y(x) = 0,
\]
see \cite[\S 22.6]{AbrSte:1964} or \cite[\S 18.8]{NIST:DLMF}. Differentiating this equation $m$ times, we find that $T_n^{(m)}$ satisfies the closely related differential equation
\[
  (1 - x^2) y''(x) - (2m+1)x y'(x) + (n^2 - m^2) y(x) = 0.
\]
Multiplying both sides by $(1 - x^2)^{m-1/2}$ yields
\[
   \big( (1 - x^2)^{m+1/2} T_n^{(m+1)}(x) \big)'  = -(n^2 - m^2)(1 - x^2)^{m-1/2} T_n^{(m)}(x).
\]
Hence,
\[
\big(\varrho_{m+1}(s) \widetilde{T}_n^{(m+1)}(s)\big)' =   -(n^2 - m^2) \varrho_{m}(s) \widetilde{T}_n^{(m)}(s).
\]
The result follows readily from this property  by applying integration by parts to the first integral in the lemma's statement, noting that $\varrho_{m}(0) = \varrho_{m}(2) = 0$ for $m \ge 1$.
\end{proof}

We point out that the central argument in the preceding lemma is a special case of the recurrence relation for derivatives of Jacobi polynomials   (cf. \cite[eq. 18.9.15]{NIST:DLMF}). However, we prefer to provide a proof here for the sake of completeness.

\begin{proposition}\label{prop:Cheby}
For any positive integer $m$,
\[
  \|f^{(m)}\| _{\varrho_{m}}\leq  \|f-S_{m-1}f\| _{m}\leq  \Bigg[\frac{2\,m^{2m}}{(2m)!}\Bigg]^{1/2} \|f^{(m)}\| _{\varrho_{m}},
\]
where
\begin{equation}\label{eq:defSm}
 S_{m} f:=\frac{1}{2}\widehat{f}(0)+\sum_{n=1}^m \widehat{f}(n) \widetilde{T}_n
\end{equation}
is the $m$th Fourier sum in the basis $\{\widetilde{T}_n\}_{n\ge 0}$.
\end{proposition}
\begin{proof}
For $m=0$, there is nothing to prove. Assume then that $m\ge 1$.
By successively applying the previous lemma, we obtain
\[
\int_{0}^2 f^{(m)}(s) \widetilde{T}_{n}^{(m)}(s)\varrho_{m}(s)\,{\rm d}s = c_{m,n}^2 \widehat{f}(n),\quad \forall n\ge m,
\]
where
\[
 c_{m,n}:= \bigg[\prod_{j=0}^{m-1} (n^2-j^2)\bigg]^{1/2} = \bigg[
 \frac{(n+m-1)!n}{(n-m)!}\bigg]^{1/2}.
\]
In particular, it follows that $\{c_{m,n}^{-1}\sqrt{\frac{2}{\pi}}\widetilde{T}^{(m)}_n\}_{n\ge m}$ forms an orthonormal basis of $L^2_{\varrho_{m}}$.

Then,
\begin{eqnarray}
\|f^{(m)}\|_{\varrho_{m}}^2 &=& \frac{\pi}{2}\sum_{n=m}^\infty c_{m,n}^{-2}\bigg[ \int_{0}^2
f^{(m)}(s) \widetilde{T}^{(m)}_{n}(s) \varrho_{m}(s)\,{\rm d}s\bigg]^2 = \frac{\pi}{2} \sum_{n=m}^\infty c_{m,n}^{2} |\widehat{f}(n)|^2. \label{eq:03:prop:Cheby}
\end{eqnarray}
The result now follows from the inequalities
\[
 1\ge  n^{-2m}c_{m,n}^ 2\ge  n^{-2m}c_{m,m}^ 2 = \frac{(2m)!}{2\,m^{2m}}.
\]
\end{proof}

For the next result, recall that we define the interpolation error as $e_L = f - Q_Lf$, as introduced in \eqref{eq:def:eL}.

\begin{proposition}\label{prop:est}
For all $m\ge 1$, there exists a constant $C_{m}$ such that for all $L\ge m-1$ it holds
\begin{equation}\label{eq:est02}
\|e_L\|_{\varrho_0}+L^{-1}\|e'_L\|_{\varrho_1}\leq C_{m} L^{-m}\|f^{(m)}\|_{\varrho_{m}}.
\end{equation}
Furthermore, for any $m\ge m_0>  5/2$, and again for any $L\ge m-1$, there exists a constant $C_{m_0,m}>0$ such that
\begin{equation}\label{eq:01:lemma:02:CleCur}
\|e'_L\|_{L^\infty(0,2)}\leq C_{m_0,m} L^{m_0-m}\|f^{(m)}\|_{\varrho_{m}}
\end{equation}
and
\begin{equation}\label{eq:02:lemma:02:CleCur}
 \|e''_L\varrho_1\|_{L^\infty(0,2)}\leq C_{m_0,m} L^{m_0-m}\|f^{(m+1)}\|_{\varrho_{m+1}}.
\end{equation}
\end{proposition}
\begin{proof}
In this proof only, we denote
\[
 f_\# = f(1+\cos\theta).
\]
We first observe that since $(\widetilde{T}_j)_{\#}(\theta)=\cos j\theta$ it holds
\begin{equation}\label{eq:equality_norms}
 \|f\|_{m} = \|f_{\#}\|_{H_\#^ m},
\end{equation}
where
\[
 \|\varphi\|_{H_\#^r}^2=   \pi
 \sum_{n=-\infty}^ \infty \max\{1,|n|^{2r}\}|\widetilde\varphi(n)|^ 2,\quad
 \widetilde\varphi(n) =\frac{1}{2\pi}\int_{-\pi}^\pi \varphi(\theta)\exp(-{\rm i}n \theta)\,{\rm d}\theta
\]
is the periodic Sobolev norm of order $r$.

Then, by Proposition \ref{prop:Cheby}, for $k=0,1$, it holds that
\begin{equation}\label{eq:Qlf}
\begin{aligned}
  \|e^{(k)}_L\|_{\rho_k} &= \| (I-S_{k-1})(Q_{L} f  -f) \|_{k}\leq \| Q_{L} f  -f  \|_{k}
  \\
  &=  \| Q_{L} (f-S_Lf)  -(f-S_Lf)  \|_{k} =  \|Q_{L,\#} (f-S_Lf)_\# -(f-S_Lf)_\#\|_{H_\#^ k},
  \end{aligned}
\end{equation}
where
\[
 \begin{array}{l}
  Q_{L,\#} \varphi\in {\rm span}\{\cos r \theta,\sin s\theta\   :\  r=0,\ldots,L,\ s=1,\ldots,L-1\},   \\[1.2ex]
   (Q_{L,\#} \varphi) (j\pi/L) = \varphi (j\pi/L),\quad j\in\mathbb{Z},
   \end{array}
\]
is the trigonometric interpolant on the periodic uniform grid $\{j\pi/L\}$. (Notice that $g_\#$ is an even function for any $g$).

From \cite[Th 8.3.1]{SaVa:2002}, for any pair $p\ge q\ge 0$ with $p>1/2$, there exists a constant $C_{p,q}>0$ such that
\begin{equation}\label{eq:Qlf:02}
\|Q_{L,\#} \varphi-\varphi\|_{{H_\#^ q}}\leq C_{p,q}  L^{q-p}\|\varphi\|_{{H_\#^p}}.
\end{equation}
Applying estimate \eqref{eq:Qlf:02} to $\varphi:= (f-S_Lf)_\#$ with $q=k\in\{ 0,1\}$ and $p=m\leq L+1$, and using Proposition \ref{prop:Cheby} again, we obtain
\begin{equation}
 \begin{aligned}
\|e^{(k)}_L\|_{\varrho_k}
\leq\ &  C_{k,m} L^{k-m}\|(f-S_L f)_\#\|_{{H_\#^ m}}
= C_{k,m} L^{k-m}\|f-S_L f\|_{m}\\
\leq\ &   C _{k,m} L^{k-m}\| f-S_{m-1}f  \|_{m}\leq C'_{k,m} L^{k-m}\|f^{(m)}\|_{\varrho_{m}},
\end{aligned}
\end{equation}
which completes the proof of the first and second estimate in \eqref{eq:est02}.

On the other hand, recalling that $T_j(\cos\theta) = \cos (j\theta)$, we have
\begin{equation}\label{exp:Tj}
\begin{aligned}
\frac{1}{j} T_j'(\cos \theta) &= \frac{\sin j\theta}{\sin\theta} =  e^{{\rm i} (j-1) \theta} \:\frac{1-e^{-2{\rm i} j\theta  }}{1-e^{-2{\rm i}  \theta} }
 = \sum_{\ell =0}^{j-1} \cos((j-2\ell-1)\theta),
\end{aligned}
\end{equation}
which implies by taking $\theta =0$ that
\[
 \|T_j'\|_{L^\infty(-1,1)}\leq |T_j'(\pm 1)| = j^2.
\]
Estimate \eqref{eq:01:lemma:02:CleCur} can be then shown to be consequence of the following argument: for $m_0\in (5/2,3]$
\begin{eqnarray*}
 \|e'_L\|_{L^\infty(0,2)} &\leq& \sum_{j =1}^\infty |\widehat{e}_L(j)|  \|T'_j\|_{L^\infty(-1,1)}\leq
  \sum_{j =1}^\infty |j|^2 |\widehat{e}_L(j)|  \\
  &\leq&  \bigg[\sum_{j=1}^\infty \frac{1}{|j|^{2m_0-4}}\bigg]^{1/2}
  \bigg[\sum_{j=1}^\infty |j|^{2m_0} |\widehat{e}_L(j)|^2\bigg]^{1/2} \\
  &\leq& C_{m_0}\|e_L \|_{m_0}=C_{m_0}\|Q_{L}  (f-S_L f)-  (f-S_Lf) \|_{m_0}\\
  &=&  C_{m_0}\|Q_{L,\#}  (f-S_L f)_\#-  (f-S_L f)_{\#}\|_{H_{\#}^{m_0}}
   \leq  C_{m_0,m} L^{m-m_0}\|   f-S_L f \|_{m}\\
   &\leq& C'_{m_0,m} L^{m-m_0}\|   f^{(m)}\|_{\rho_m}
\end{eqnarray*}
where we have applied sequentially Cauchy-Schwarz inequality, \eqref{eq:equality_norms}, estimate \eqref{eq:Qlf:02} and Proposition \ref{prop:Cheby} (using again that $L\geq m-1$) to bound the last term.

The last estimate \eqref{eq:02:lemma:02:CleCur} follows from a very similar argument, using now
\[
  \|e''_L\varrho_1\|_{L^\infty(0,2)}  \leq \sum_{j =2}^\infty |\widehat{e}_L(j)|  \|\varrho_1\widetilde{T}''_j\|_{L^\infty(0,2)}
\]
and the bound
\[
 \|\varrho_1 \widetilde{T}_j''\|_{L^\infty(0,2)} \leq \frac{1}{2} j^3.
\]
which is a consequence of
\[
\sin\theta \, T_j''(\theta) = -j \frac{\rm d}{{\rm d}\theta} \left( \frac{\sin j\theta}{\sin\theta} \right)
\]
and \eqref{exp:Tj} (see also \cite[Lemma 2.2]{Dom:2014}).

\end{proof}

\begin{lemma}\label{lemma:03:CleCur}
There exists $C_1,\ C_0>0$ such that all $\nu {>0}$
\begin{equation}\label{eq:01:lemma:03:CleCur}
 \int_0^2 \exp(-\nu s)\,\varrho_1(s)\, {\rm d}s \leq   C_1 \nu^{-3/2},\quad
\end{equation}
and
\begin{equation}\label{eq:02:lemma:03:CleCur}
 \int_0^2  \exp(-\nu s)\,\varrho_{0}(s)\,{\rm d}s \leq C_0 \nu^{-1/2}.
\end{equation}
\end{lemma}
\begin{proof}
{Using the change of variables $t=\nu s$, we obtain}
\[
\begin{aligned}
\int_0^2 \exp(-\nu s)\,\varrho_1(s)\, {\rm d}s &\ =\
\nu^{-1}\int_0^{2\nu} \exp(-t)\,\varrho_1(t/\nu)\, {\rm d}t\le
\sqrt{2}\nu^{-3/2}\int_0^\infty \exp(-t)\,t^{1/2}\, {\rm d}t \\
&\ = \sqrt{2}\Gamma(3/2)\nu^{-3/2}=\frac{1}{\sqrt{2}}\pi \nu^{-3/2}
\end{aligned}
\]
{which proves the first result }

{For the second result, we can simply proceed as follows }
\[
\begin{aligned}
 \int_0^2 \exp(-\nu s)\,\varrho_0(s)\, {\rm d}s \ &\le\  \int_0^1 e^{-\nu s} s^{-1/2}\,{\rm d}s + e^{-\nu} \int_1^{2}  (2-s)^{-1/2}\,{\rm d}s\\
 &\le\  \nu^{-1/2}\int_0^\infty e^{-t} t^{-1/2}\,{\rm d}t + 2 e^{-\nu}\le  \left(\Gamma(1/2)+2\max_{\nu \ge 0} e^{-\nu}\nu^{1/2}\right)\nu^{-1/2}
\\
 &\le\  (\sqrt{\pi}+\sqrt{2}e^{-1/2}) \nu^{-1/2}.
 \end{aligned}
\]

\end{proof}

We are now ready to present the main result of this section, which presents different estimates for the error of the quadrature rule. These estimates account for $L$, the number of quadrature points, and varying ranges of $ \sigma = \Re z$, which dictates the exponential decay, and $|z|$, which also captures the potential oscillatory behavior.

Recall the definition of $\mathbb{C}_{\mu}$ given in \eqref{eq:defCmu}.
\begin{theorem}\label{theo:CleCur}
For all $\mu_0>0$ and $m\geq 1$, there exists a constant $C_{m,\mu_0}>0$ such that for all $L\geq m-1$
and $z\in\mathbb{C}_{\mu_0}$ with $\sigma = \Re z$, the following holds:
 \begin{equation}\label{eq:01:theo:CleCur}
  |I_{L,z}(f)-I_{z}(f)| \leq C_{m,\mu_0} L^{1-m} (1+|\sigma|)^{-3/4} \min\bigg\{\frac{1}{L},\frac{1}{|z|}(1+|\sigma|)^{1/2} \bigg\}\|f^{(m)}\|_{\varrho_{m}}.
\end{equation}
Moreover, for all $m\geq m_0>7/2$,
there exists a constant $C'_{m,m_0,\mu_0}$, depending on $m$, $m_0$ and $\mu_0$ but independent of $L\geq m-1$ and  $z\in\mathbb{C}_{\mu_0}$ such that
 \begin{equation}\label{eq:02:theo:CleCur}
   |I_{L,z}(f)-I_{z}(f)|\leq C'_{m,m_0,\mu_0} (1+|z|)^{-2} L^{m_0-m-1}\bigg(
   1+ (1+|\sigma|)^{-1/2} L
   \bigg)\|f^{(m)}\|_{\varrho_{m}} .
\end{equation}
\end{theorem}
\begin{proof}
Since $\varrho_0\varrho_1\equiv 1$, by applying the Cauchy-Schwarz inequality, we obtain:
\begin{equation}\label{eq:03:theo:CleCur}
 |I_{L,z}(f)-I_{z}(f)|= \bigg|\int_0^2 e_L(s)\exp(z s)\,{\rm d}s\bigg|\leq  \|e_L\|_{\varrho_0}  \|\exp(\sigma\,\cdot\,)\|_{\rho_1}.
\end{equation}
Alternatively, applying integration by parts and noting that $e_L(0)=e_L(2)=0$, we obtain
\begin{eqnarray}
|I_{L,z}(f)-I_{z}(f)|& = &
\frac{1}{|z|}\bigg|\int_0^2 e'_L(s)\exp(z s)\,{\rm d}s\bigg|
\leq \frac{1}{|z|}
\|e'_L\|_{\varrho_1}  \|\exp(\sigma\,\cdot\,)\|_{\varrho_0}.
\label{eq:04:theo:CleCur}
\end{eqnarray}
Estimate \eqref{eq:01:theo:CleCur} for $\sigma \geq -1/2$ follows directly from \eqref{eq:03:theo:CleCur}--\eqref{eq:04:theo:CleCur} and \eqref{eq:est02} in Proposition \ref{prop:est}, whereas the case $\sigma \leq -1/2$ additionally requires applying Lemma \ref{lemma:03:CleCur} with $\nu = -2\sigma$, which implies
\[
|\sigma^{3/4}|\|\exp(\sigma\,\cdot\,)\|_{\varrho_1} +
|\sigma^{1/4}|\|\exp(\sigma\,\cdot\,)\|_{\varrho_0} \leq C.
\]

To prove \eqref{eq:02:theo:CleCur} we can assume again that $\sigma\le -1/2$, apply an additional step of integration by parts to bound
\begin{eqnarray*}
|I_{L,z}(f)-I_{z}(f)|&\leq& \frac{1}{|z|^2}  \bigg| e_L'(s) \exp(sz)\Big|_{s=0}^{s=2}\bigg|
+\frac{1}{|z|^2}\int_0^2 |e''_L(s)|\exp(\sigma  s)\,{\rm d}s\\
&\leq& \frac{1}{|z|^2}\Big[ 2\|e_L'\|_{L^\infty(0,2)}+\|\varrho_1 e''_{L}\|_{L^\infty(0,2)}
\|\varrho_0 \exp(\sigma \,\cdot\,)\|_{L^1(0,2)} \Big]
\end{eqnarray*}
and use \eqref{eq:01:lemma:02:CleCur} and \eqref{eq:02:lemma:02:CleCur}
from Proposition \ref{prop:est} and \eqref{eq:02:lemma:03:CleCur} from Lemma \ref{lemma:03:CleCur}.

\end{proof}

\section{Implementation}

As discussed in  previous sections, implementing the rule  {$I_{L,z}(f)$ } requires a fast and numerically stable computation of the weights \added[id=R2]{\eqref{eq:thequadformula:b}}.
%\[
% \omega_n(z):= \int_{0}^2 T_{n}(s-1)\exp(z s)\,{\rm d}s, \quad { n=0,\ldots,L.}
%\]
We consider the closely related coefficients
\begin{equation}\label{eq:rhon}
 \rho_n(z) := \int_{0}^2 U_{n}(s-1)\exp(z s)\,{\rm d}s, \qquad U_n(s) =\frac{1}{n+1}T_{n+1}'(s).
\end{equation}
(Thus, $U_n$ is the Chebyshev polynomial of the second kind). By integrating by parts and using the identity  $T_n(-1)=(-1)^n$, we establish the relationship  between the two weight  sequences:
\begin{equation}\label{eq:gamman}
 \omega_{n+1}(z) = \gamma_{n+1}(z) - \frac{n+1}{z} \rho_n(z),\quad \gamma_n(z):=
 \begin{cases}
       \frac{1}{z}(e^{2z}-1),& \text{if $n$ is even },\\
       \frac{1}{z}(e^{2z}+1),& \text{if $n$ is odd }.
 \end{cases}
\end{equation}
Furthermore, since
\[
 U_{n+1}(x) -U_{n-1}(x) = 2T_{n+1}(x),
\]
(consequence of the trigonometric identity $ \sin(n+2)\theta- \sin n\theta = 2\sin\theta \cos(n+1)\theta$)
we obtain
\[
 \rho_{n+1}(z) - \rho_{n-1}(z) = 2 \omega_{n+1}(z) = 2\gamma_{n}(z) - \frac{2(n+1)}{z}\rho_n(z),
\]
which provides the following recurrence relation for the coefficients, which serves as the key element of the algorithm:
\begin{equation}\label{eq:rho:01}
 -\rho_{n-1}(z) + \frac{2n+2}{z} \rho_n(z) + \rho_{n+1}(z) = 2\gamma_{{n}}(z).
\end{equation}

These recurrence relations have been used in the literature to implement this and related quadrature rules, for instance, in Filon-Clenshaw-Curtis rules for smooth integrands \cite{QUADPACK,DoGrSm:2010,Xiang2014} and for integrands exhibiting logarithmic singularities \cite{Dom:2014}. Moreover, \cite{Bakhvalov1968} applies a nearly identical property to Legendre polynomials in the context of Filon-Gaussian rules.

{This yields an initial algorithm for computing the weights $\omega_n(z)$ (and $\rho_n(z)$ as well) for $n=0,\ldots,L_0$, where $L_0$ must be relatively small—below a threshold $n_0 = n_0(z)$ to be specified later (see \eqref{eq:n0} below). Therefore, we set $L_0 = \min\{L, n_0\}$, where $L$ is the number of weights required by the rule. In other words, this initial algorithm may already be sufficient when $L$ is small and/or $z$ is large, as will become clear shortly.}

{However, as will be shown, this may not be sufficient even for moderate values of $L$ if $z$ is large enough, and thus the calculation of the weights becomes more involved. The idea is then to rewrite \eqref{eq:rho:01} as a tridiagonal linear system. This is a classical approach that can be traced back to the works of  Oliver \cite{Oliver1968} and Olver \cite{Olver1967}. This leads to Algorithm~2, or Phase~2 of the full procedure. This reformulation requires the value of $\rho_{m}(z)$ for some (typically large) index $m$. The computation of this extreme coefficient is the subject of the third algorithm, or Phase~3 of the full method.}

{The full algorithm therefore consists of up to three phases. The first, which applies to the initial coefficients, is based on the recurrence relation. The second, which computes the (possibly) remaining weights, involves solving a tridiagonal linear system. The third provides the value of $\rho_{m}(z)$ required in Phase~2. We derive these components below, establishing numerical stability and/or convergence along the way. The complete algorithm is given in Algorithm~\ref{alg:04}, while the individual components are presented in Algorithms~\ref{alg:01}--\ref{alg:03}.}

\subsection{Phase 1: {Computation of $\omega_n(z)$ and $\rho_n(z)$ for $n=0,\ldots,L_0$, with $L_0 \le n_0$}}

Identity \ref{eq:gamman} and recurrence relation \eqref{eq:rho:01} provide a straightforward but na\"ive method for computing the desired weights  $\bm{\omega}_{{L_0}} (z) := (\omega_0(z), \ldots, \omega_{L_0}(z))$ from $\bm{\rho}_{L_0}(z) := (\rho_0(z), \ldots, \rho_{{L_0}}(z))$. This results in  Algorithm \ref{alg:01}.
\begin{algorithm}
\begin{algorithmic}[1]
       \Require $L_0$, $z$, and $\gamma_n(z)$ as defined in \eqref{eq:gamman}.
       \State $\displaystyle \rho_0(z) = \frac{\exp(2z)-1}{z}$
       \State $\displaystyle \rho_1(z) =  \frac{2(z + \exp(2z)(z - 1) + 1)}{z^2}$
       \State $\displaystyle \omega_0(z) = \rho_0(z)$
       \State $\displaystyle \omega_1(z) = \tfrac12\rho_1(z)$
        \For{$n = 1:L_0-1$}
            \State $\displaystyle\omega_{n+1}(z) =  - \frac{n+1}{z} \rho_n(z)+ \gamma_{n+1}(z)$
            \State $\displaystyle\rho_{n+1}(z) = \rho_{n-1}(z) - \frac{2n+2}{z} \rho_n(z) + 2\gamma_{n}(z)$
         \EndFor
        \State \Return $\bm{\omega}_{L_0}(z) = (\omega_0(z),\ldots, \omega_{L_0}(z)),\ \bm{\rho}_{L_0}(z)=(\rho_0(z),\ldots, \rho_{L_0}(z))$
\end{algorithmic}
\caption{\label{alg:01}Computation of the weight vectors $\bm{\omega}_{L_0}(z)$ and $\bm{\rho}_{L_0}(z)$.}
\end{algorithm}

Unfortunately, it is well known (cf. \cite{DoGrSm:2010,Pi:2000}) that this recurrence relation is stable for $\sigma = \Re z=0$ only when $n\lesssim |z|$. This was rigorously proven in that paper by analyzing the recurrence
\begin{equation}\label{eq:recurrenceSolutions}
\begin{array}{l}
 \delta_{m-1}^{(m)} =0,\quad
 \delta_{m}^{(m)}   =1,\\
 \delta_{j+1}^{(m)} = \delta_{j-1}^{(m)} - \frac{2j}{z} \delta_j^{(m)}, \quad j = m, \ m+1, \ldots
 \end{array}
\end{equation}
which illustrates how errors in evaluating of \eqref{eq:rho:01} propagate along this sequence. The solution is given by
\[
 \delta_j(z)= -\frac{1}{2} \pi  z (-i)^{m+j+1} \left[Y_{m-1}(i z) J_j(i z)-J_{m-1}(i z) Y_j(i z)\right],
\]
where $J_j$ and $Y_j$ denote the Bessel functions of the first and second kind, respectively. A detailed {analysis} involving accurate bounds for the Bessel functions on the real line established the result for $z \in \mathrm{i}\mathbb{R}$, concluding that Algorithm \ref{alg:01} is effectively stable in this particular case, when $|z| = |\Im z|\lesssim n$. However, extending this {analysis} to general complex numbers $z$, i.e., when $\Re z\neq 0$, is challenging due to the significantly more intricate behavior of the Bessel functions. Nonetheless, numerical experiments have demonstrated that instability in the recurrence relations arises even earlier when $\sigma =\Re z \neq 0$, with the worst-case scenario occurring for real $z < 0$, where stability clearly breaks down as $n\approx |z|^{1/2}$. This result is proven in the next proposition.

The extent to which the stability range of the recurrence relation extends as the ratio $|\Im z|/|z|$ increases, and consequently how the usability of Algorithm \ref{alg:01} is affected as the imaginary part of $z$ gains weight, remains an open problem. We refer the reader to next section for some numerical experiments on this topic (see Figure \ref{fig:02}).

\begin{proposition}\label{prop:stab}
Let $(\rho_n(z))_{n\geq 0}$ \added[id=R2]{be} the sequence of weights \eqref{eq:rhon} and $(\widetilde{\rho}_n(z))_{n\geq 0}$ \added[id=R2]{be} that given by a perturbed version of Algorithm \ref{alg:01}
\begin{equation}\label{eq:01:prop:stab}
\begin{aligned}
\widetilde{\rho}_0(z)\ &=\ \frac{\exp(2z)-1}{z} +\varepsilon_0\\
\widetilde{\rho}_1(z)\ &=\ \frac{2(z + \exp(2z)(z - 1) + 1)}{z^2} +\varepsilon_1\\
\widetilde{\rho}_{j+1}(z) \ &=\ \widetilde{\rho}_{j-1}(z) - \frac{2j+2}{z} \widetilde{\rho}_j(z) +\gamma_j(z)+ \varepsilon_{j+1},\quad j = 2,3,\ldots
\end{aligned}
\end{equation}
with
\[
 |\varepsilon_j|\leq   \varepsilon.
\]
Then for any $|z|$ sufficiently large and for $n+1<|z|$,
\[
|{\rho}_n(z)-\widetilde{\rho}_n(z)| <e^{(n+1)^2/|z|} \left(\left\lceil n/2\right\rceil +1\right)\varepsilon.
\]

\end{proposition}
\begin{proof}Define for $m = 0, 1, \ldots,$
\begin{equation}\label{eq:recurrenceSolutions:02}
\begin{array}{l}
 \varepsilon_{2m}^{(m)} = \varepsilon,\quad
 \varepsilon_{2m+1}^{(m)} = \varepsilon,\\
 \displaystyle \varepsilon_{j+1}^{(m)} = \varepsilon_{j-1}^{(m)} + \frac{2(j+1)}{|z|}\varepsilon_j^{(m)}, \quad j = 2m+1, \ 2m+2, \ldots.
 \end{array}
\end{equation}
Clearly, $(\varepsilon^{(m)})_{n\geq 2m}$ is an increasing sequence that can be used to derive the following bounds for the difference between the original and the perturbed sequence:
\begin{equation}\label{eq:wagner:parsifal:01}
|{\rho}_{2n}(z) - \widetilde{\rho}_{2n}(z)| \leq \varepsilon + \sum_{m=0}^{n-1} \varepsilon_{2n}^{(m)}, \quad
|{\rho}_{2n+1}(z) - \widetilde{\rho}_{2n+1}(z)| \leq 2\varepsilon + \sum_{m=0}^{n-1} \varepsilon_{2n+1}^{(m)}.
\end{equation}
% \[
% |{\rho}_{2n}(z) - \widetilde{\rho}_{2n}(z)| \leq \varepsilon + \sum_{m=0}^{n-1} \varepsilon_{2n}^{(m)}, \quad
% |{\rho}_{2n+1}(z) - \widetilde{\rho}_{2n+1}(z)| \leq 2\varepsilon + \sum_{m=0}^{n-1} \varepsilon_{2n+1}^{(m)}.
% \]
Let us now focus on bounding
\begin{equation}\label{eq:E2nE2n+1}
 E_{2n} = \sum_{m=0}^{n-1} \varepsilon_{2n}^{(m)},\quad E_{2n+1} = \sum_{m=0}^{n-1} \varepsilon_{2n+1}^{(m)},
\end{equation}
from which the result will follow.

It is straightforward to show by induction that
\[
|\varepsilon_{r}^{(m)}| \leq \prod_{j=2m+2}^{r}\left(1+\frac{2j}{|z|}\right)\varepsilon
\leq \exp\left(\frac{\alpha_r-\alpha_{2m+1}}{|z|}\right),\quad \alpha_k := 2\sum_{\ell=1}^{k} \ell = k(k+1).
\]
This implies
\begin{eqnarray*}
E_{2n} &\leq& e^{4n(n+1)/|z|} \left[\sum_{m=0}^{n-1} e^{-2m(2m+1)/|z|}\right] \varepsilon
 \leq  e^{4(n+1)^2/|z|} n  \varepsilon .
\end{eqnarray*}
%with
% \[
%  C_{n-1}(z):= \left[\sum_{m=0}^{n-1} e^{-4m^2/|z|}\right].
% \]
% Furthermore,
% \[
%  C_{n-1}(z) \leq   1 + \int_{0}^{n-1} e^{-4x^2/|z|} \,{\rm d}x
% \leq  1 + \frac{1}{4}\sqrt{\pi |z|} \mathop{\rm erf}\left(\frac{2(n-1)}{\sqrt{|z|}}\right) .
% \]
% Using the inequality
% \[
% \frac{1}{2}\sqrt{\pi} \mathop{\rm erf}(x) \leq x, \quad x \ge 0,
% \]
% we obtain
% \[
%  C_{n-1}(z) \leq  n.
% \]
% Thus, we have proved
% \[
% E_{2n} \leq e^{4(n+1)^2/|z|} n\varepsilon,
% \]
Substituting this into the first inequality in \eqref{eq:wagner:parsifal:01}, we get
\[
|{\rho}_{2n}(z) - \widetilde{\rho}_{2n}(z)| < \varepsilon + \sum_{m=0}^{n-1} \varepsilon_{2n}^{(m)}
\leq (1 + e^{4(n+1)^2/|z|}n)\varepsilon
< e^{4(n+1)^2/|z|}(n+1)\varepsilon.
\]
Similarly, we can prove
\[
E_{2n+1} \leq   e^{(n+2)^2/|z|} n \varepsilon %`\le  e^{4(n+2)^2/|z|}(n+2)\varepsilon,
\]
and the second inequality in \eqref{eq:wagner:parsifal:01} yields
\[
|{\rho}_{2n+1}(z) - \widetilde{\rho}_{2n+1}(z)| \le 2\varepsilon+ E_{2n+1}\le   2\varepsilon+e^{4(n+2)^2/|z|} n \varepsilon  < e^{4(n+2)^2/|z|}(n+2)\varepsilon.
\]
The result follows now readily.
\end{proof}

\begin{remark}

This proposition implies the stability of the direct recurrence-relation algorithm (Algorithm~\ref{alg:01}) for large values of $|z|$ and/or moderate values of $L$.
Numerical experiments indicate that beyond $\big\lceil 2|z|^{1/2} \big\rceil + 1$  the computation can become effectively unstable. In particular, for $z = \sigma < 0$ (purely exponential-decay integrals) and $n \approx 4|\sigma|^{1/2}$, the instability significantly contaminates the numerical evaluation, rendering the results meaningless.
For $z = \eta{\rm i} \in {\rm i}\mathbb{R}$, as commented above, we can safely set $n_0 = \lceil |z| \rceil + 1$ following the results established in \cite{DoGrSm:2010}, which extends the range of applicability of this algorithm in the purely oscillatory case. {Defining a sharper value of $n_0$, which should depend on the real and imaginary parts of $z$ rather than on $|z|$, and which provides a reliable and precise stability threshold for Algorithm~\ref{alg:01}, remains an open problem.}

In any case, a transition region exists—for $L \gtrsim |z|^{1/2}$ in the general case, or $L \gtrsim |z|$ when $z = \eta{\rm i}$—in which, even for moderate values of $L$, the quadrature is not applicable if one relies solely on Algorithm~\ref{alg:01}. The next subsection discusses how to perform these calculations in this range of $z$.

\end{remark}

%
%
%    rho(2*j-1) = -2*(2*j-2)/z*rho(2*j-2)+rho(2*j-3)   +2*(exp(2*z) -1)/z;
%    rho(2*j)   = -2*(2*j-1)/z*rho(2*j-1)+rho(2*j-2)   +2*(exp(2*z) +1)/z;

%\subsection{Phase 2: $|n|\ge n_0$}

\subsection{Phase 2: {Computation of $\omega_n(z)$ and $\rho_n(z)$ for $n=n_0+1,\ldots..., L$}}
This subsection focuses on the evaluation of the weights $\rho_n(z)$ and $\omega_n(z)$ when $|n| \gtrsim |z|^{1/2}$. As previously mentioned, in our implementation, we have set the threshold at
\begin{equation}\label{eq:n0}
n_0 = n_0(z) := \begin{cases}
\left\lceil 2|z|^{1/2}\right\rceil+1, &\text{if } \sigma \ne 0,\\
\left\lceil  |z|\right\rceil+1, &\text{if } \sigma= 0,
                \end{cases}\quad  \text{with }\sigma =\Re z.
\end{equation}
Thus, after Phase 1, we have already computed $(\rho_0(z), \rho_1(z), \ldots, \rho_{n_0}(z))$.

{For the time being take $n_1\ge L$. Using again \eqref{eq:rho:01}, we conclude that
\[
 \bm{\rho}_{\bm{n}}(z) := (\rho_{n_0+1}(z), \rho_{n_0+2}(z), \ldots, \rho_{n_1}(z)) \in \mathbb{C}^{n}
\]
where
\begin{equation}\label{eq:n}
\bm{n} := (n_0, n_1), \quad n := n_1 - n_0.
\end{equation}
solves  a  tridiagonal linear system
\begin{equation}\label{eq:linear:system:phase2}
\mathrm{A}_{\bm{n}}(z)\bm{\rho}_{\bm{n}}(z) =  {\bf b}_{\bm{n}}(z),
\end{equation}
with   the right-hand side  given by
\begin{equation}\label{eq:bs}
{\bf b}_{\bm{n}}(z) := (b_1(z), b_2(z), \ldots, b_n(z))^\top, \quad
{b}_j(z) :=
\begin{cases}
        2\gamma_{n_0+2}(z) + \rho_{n_0}(z), & \text{if $j=1$}, \\
       2\gamma_{n_0+j+1}(z), & \text{if $j=2, 3, \ldots, n-1$}, \\
        2\gamma_{n_1+1}(z) - \rho_{n_1+1}(z), & \text{if $j=n$}.
\end{cases}
\end{equation}
Furthermore,  the matrix of this system can be written
\begin{equation}\label{eq:An}
\mathrm{A}_{\bm{n}}(z) :=
 \frac{2}{z} \mathrm{D}^{1/2}_{\bm{n}}
 \left(
 \mathrm{I}_n + \tfrac{1}{2} z \mathrm{M}_{\bm{n}}
 \right) \mathrm{D}^{1/2}_{\bm{n}},
\end{equation}
with $\mathrm{I}_n$ denoting the identity matrix of order $n$, and
\begin{eqnarray}
\mathrm{D}_{\bm{n}} &:=&
  \begin{bmatrix}
   n_0+2 & \\
   & n_0+3 & \\
   && \ddots & \\
   &&& n_1+1
 \end{bmatrix}, \label{eq:Dn}\\
 \mathrm{M}_{\bm{n}} &:=&
  \begin{bmatrix}
    & \frac{1}{\sqrt{(n_0+2)(n_0+3)}} \\
    -\frac{1}{\sqrt{(n_0+2)(n_0+3)}} & & \frac{1}{\sqrt{(n_0+3)(n_0+4)}} \\
    & -\frac{1}{\sqrt{(n_0+3)(n_0+4)}} & & \ddots \\
    && \ddots & \frac{1}{\sqrt{n_1(n_1+1)}} \\
    && -\frac{1}{\sqrt{n_1(n_1+1)}}&
\end{bmatrix}. \label{eq:Mn}
\end{eqnarray}
}

Observe that $\mathrm{M}_{\bm{n}}$ is a skew-symmetric matrix. Consequently, the spectrum of $\mathrm{M}_{\bm{n}}$, denoted by $\sigma(\mathrm{M}_{\bm{n}})$, consists of purely imaginary numbers.

The tridiagonal system, cf. \eqref{eq:tridiagonalsystem}, can be solved in ${\cal O}(n)$ operations using the Thomas algorithm. The resulting algorithm, presented in Algorithm \ref{alg:02}, requires
%$\rho_{n_0+1}(z)$ and $\rho_{m+1}(z)$
{$\rho_{n_0}(z)$ and $\rho_{n_1+1}(z)$} {(see \eqref{eq:bs})}
and computes $n_1 - n_0$ coefficients, although only $L$ coefficients are returned. These requirements are interconnected. First, $\rho_{n_0+1}(z)$, along with all coefficients $\rho_n(z)$ and $\omega_n(z)$ for $n$ up to $n_0+1$, can be computed in the first algorithm, i.e., Phase 1 of the full procedure. The other term, $\rho_{n_1+1}(z)$, is a more delicate matter, and its calculation will be discussed in next subsection. {In few words, we will show that $\rho_{m}(z)$  be computed for $m$ {\em sufficiently large}, which will ultimately provide the coefficient $\rho_{n_1+1}(z)$ (see Algorithm \ref{alg:03}). Consequently, it is quite possible—and indeed almost necessary (as will be clarified in the next subsection)—to work with $n_1 \ge L$, which entails computing additional coefficients beyond those strictly required.
}

\begin{algorithm}
\begin{algorithmic}[1]
    \Require $L$, $z$, $\rho_{n_0}(z)$, and $\rho_{n_1+1}(z)$ for $n_0\le L-1$ and some  $n_1 \ge L$
    \State Define ${\bf b}_{\bm{n}}(z)$, $\mathrm{D}_{\bm{n}}$, and $\mathrm{M}_{\bm{n}}$ according to \eqref{eq:bs}, \eqref{eq:Dn}, and \eqref{eq:Mn}.
    \State Set ${\bf c}_{\bm{n}} = \frac{z}{2} \mathrm{D}_{\bm{n}}^{-1/2} {\bf b}_{\bm{n}}(z)$.
    \State Solve $(\mathrm{I}_n + \frac{z}2 \mathrm{M}_{\bm{n}}) {\bm y}_{\bm{n}} = {\bf c}_{\bm{n}}$  \label{eq:tridiagonalsystem}.
    \State Set $\bm{\rho}_{\bm{n}}(z) = \mathrm{D}_{\bm{n}}^{-1/2} {\bm y}_{\bm{n}}$.
    \For{$n = n_0:L-1$}
        \State $\displaystyle \omega_{n+1}(z) = - \frac{n+1}{z} \rho_{n}(z) + \gamma_{n+1}(z)$.
    \EndFor
    \State \Return $(\rho_{n_0+1}(z), \ldots, \rho_{L}(z))$, $(\omega_{n_0+1}(z), \ldots, \omega_{L}(z))$.
\end{algorithmic}
\caption{\label{alg:02}Computation of the weights $(\rho_{n_0+1}(z), \ldots, \rho_{L}(z))$ and $(\omega_{{n_0+1}}(z), \ldots, \omega_{L}(z))$.}
\end{algorithm}

We now focus on the stability of the algorithm. In particular, we will show the unique solvability of system \eqref{eq:linear:system:phase2} and provide estimates for the condition number of the matrix $\mathrm{A}_{\bm{n}}(z)$. Notice that in the next result, $\|{\rm B}\|_r$ denotes the $r$-norm of the matrix ${\rm B}$.
\begin{theorem}\label{theo:stability:phase02}
 Let $\sigma = \Re z\ne 0$ and $\eta = \Im z$. Then, the following holds:
\begin{align}
   \|(\mathrm{I}_n+\tfrac12 z \mathrm{M}_{\bm{n}})^{-1}\|_2 & \ \leq\  \begin{cases}
                                                                |z||\sigma|^{-1}, & \text{if $|\eta|\leq \frac{|z|^2}{n_0+2}$}, \\
                                                                \left(1- \frac{|\eta|}{n_0+2} \right)^{-1}, & \text{if $|\eta|\geq \frac{|z|^2}{n_0+2}$}.
     \end{cases}\label{eq:01:theo:stability:phase02}
\\
    \|(\mathrm{I}_n+\tfrac12 z \mathrm{M}_{\bm{n}})\|_r & \ \leq\ 1 +\tfrac12 |z| \|\mathrm{M}_{\bm{n}}\|_r<\left(1+\frac{|z|}{n_0+2}\right), \quad r\in [1,\infty].
    \label{eq:02:theo:stability:phase02}
\end{align}
In particular, the condition number of the matrix
$\mathrm{A}_{\bm{n}}(z)$ in Algorithm \ref{alg:02} when $n_0=n_0(z)$ is taken as in \eqref{eq:n0}   can be estimated by
\begin{equation}\label{eq:03:theo:stability:phase02}
 \kappa_2(\mathrm{A}_{\bm{n}}(z))\leq \begin{cases}
                                                 2 |\sigma^{-1}z|                & \text{if $\sigma\ne 0$}, \\
                                                    2 & \text{if $\sigma =0$}.
     \end{cases}
\end{equation}
\end{theorem}
\begin{proof}
First, observe that
\[
 \|{\rm M}_{\bm{n}}\|_{1}= \|{\rm M}_{\bm{n}}\|_{\infty}\leq \frac{2}{n_0+2}.
\]
This estimate, along with the Riesz-Thorin theorem, implies \eqref{eq:02:theo:stability:phase02}.

Let $S_{\bm{n}}(z)$ denote the set of singular values of the matrix
\[
 \mathrm{I} + \frac{1}{2} z \mathrm{M}_{\bm n}.
\]
By definition, the elements of $S_{\bm{n}}(z)$ are the positive square roots of the eigenvalues of the matrix
 \[
 \left(\mathrm{I}_n+\tfrac12 z \mathrm{M}_{\bm{n}}\right)^*
 \left(\mathrm{I}_n+\tfrac12 z \mathrm{M}_{\bm{n}}\right) = \mathrm{I}_n +\eta {\rm i}\mathrm{M}_{\bm{n}}-\frac{1}{4}|z|^2 \mathrm{M}_{\bm{n}}^2.
 \]
(Note that $\mathrm{M}_{\bm{n}}^\top =-\mathrm{M}_{\bm{n}}$.) Therefore,
\[
S_{\bm{n}}(z) = \{ ( 1-2\eta\alpha +|z|^2 \alpha^2)^{1/2} \ : \ \pm 2\alpha {\rm i}\in \rho(\mathrm{M}_{\bm{n}})\}\subset (0,\infty),
\]
where $\rho(\mathrm{M}_{\bm{n}})\subset{\rm i}\mathbb{R}$ denotes the spectrum of $\mathrm{M}_{\bm{n}}$.

By the Gershgorin Circle theorem, we have
\[
\alpha\in  \left(-\frac{1}{n_0+2},\frac{1}{n_0+2}\right),
\]
so that the constants $C(n_0)>c(n_0)$, given by
\[
\begin{aligned}
 C^2(n_0) &\ := \  \max_{x\in [-\frac{1}{n_0+2},\frac{1}{n_0+2}]} (1-\eta x) ^2+(\sigma x)^2,\\
 c^2(n_0) &\ := \  \min_{x\in [-\frac{1}{n_0+2},\frac{1}{n_0+2}]} (1-\eta x) ^2+(\sigma x)^2,
 \end{aligned}
\]
are as upper and lower bounds for $S_{\bm{n}}(z)$.

Clearly,
\[
 C^2(n_0)\leq 1+\frac{2|\eta|}{n_0+2}+\frac{|z|^2}{(n_0+2)^2}\leq \left(1+\frac{|z|}{n_0+2}\right)^2,
\]
which, in turn, provides an alternative proof of \eqref{eq:02:theo:stability:phase02} for $2-$norm.

We now focus on proving \eqref{eq:01:theo:stability:phase02}. Without loss of generality we can assume  $\eta\ge 0$. Observe that, with
\[
  f(x): =  1-2\eta x +|z|^2 x^2,
\]
it follows that
\[
 \argmin_{x\in\mathbb{R}} f(x) = \frac{\eta}{\sigma^2+\eta^2} = \frac{\eta}{|z|^2}=:x_0, \quad \min_{x\in\mathbb{R}}f(x) = f(x_0)=\frac{\sigma^2}{\sigma^2+\eta^2}=\frac{\sigma^2}{|z|^2}.
 \]
Thus,
 \[
 c^2(n_0)\geq \begin{cases}
\frac{\sigma^2}{|z|^2}, & \text{if }x_0\leq \frac{1}{n_0+2}, \\
f\left(\frac{1}{n_0+2}\right), & \text{if }x_0> \frac{1}{n_0+2}.
             \end{cases}
 \]
Furthermore,
\[
 f\left(\frac{1}{n_0+2}\right) = 1-\frac{2\eta}{n_0+2}+\frac{\sigma^2+\eta^2}{(n_0+2)^2}= \left(1-\frac{\eta}{ n_0+2}\right)^2+\frac{\sigma^2}{(n_0+2)^2},
\]
which concludes the proof.

Finally, \eqref{eq:03:theo:stability:phase02} is a straightforward consequence of  \eqref{eq:01:theo:stability:phase02} and  \eqref{eq:02:theo:stability:phase02}.
 \end{proof}
\begin{remark} Theorem \ref{theo:stability:phase02} implies the stability of the algorithm for any $z$.
Certainly, one might be concerned about the case where $0 < |\sigma| \ll |z|$ due to the factor $|z|||\sigma^{-1}|$. Several key observations can be made in this situation. First, if $\sigma$ is very small, we can incorporate this term into the function $f$ and shift all calculations to the purely oscillatory integral case. On the other hand, if the second-phase algorithm is applied in its current form, we note that the estimate of the norm $|(\mathrm{I}_n+\tfrac12 z \mathrm{M}_{\bm{n}})^{-1}|2$ depends on the distance of $z$ from the discrete set of eigenvalues of ${\mathrm{M}_{\bm n}}$. In practice, this distance is often larger than the pessimistic estimate derived in the proof. Therefore, the condition number of the matrix is, in practice, better than the bound established in the theorem for this case. Furthermore, a small increase or decrease in $n_0$ by a few units can resolve any potential ill-conditioning issues.

Finally, we emphasize that our computations never exhibited numerical instability, even in cases where the theorem suggests potential issues.
\end{remark}

\subsection{Computation of $\rho_{m }(z)$ for some $m\ge L$}

As pointed out, Algorithm \ref{alg:02} requires computing $\rho_{m}(z)$ for some $m \ge L$. Here, we describe how this computation can be performed.

Set   $ m_1 > m_0 > |z| $,
\[
\bm{m}:=(m_0,m_1),\quad m:=m_1-m_0.
\]
We assume $ m $ to be odd  so that $ m_{\rm c}: = (m_1 + m_0+1)/2 $ is an integer corresponding to the middle entry of any vector of length $m$ with indices from $m_0+1$ to $m_1$.

Following the notation of the preceding section (see \eqref{eq:linear:system:phase2}), we have
\begin{equation}\label{eq:leftovers:01}
 {\rm A}_{\bm{m}}(z){\bm \rho}_{\bm m}(z) = {\bf b}_{\bm m}(z).
\end{equation}
(Note that ${\rm A}_{\bm{m}}(z)$ is now a $m\times m$ diagonally dominant matrix and therefore invertible.)

Instead, we solve:
\begin{equation}\label{eq:46}
  {\rm A}_{\bm{m}}(z)\widetilde{\bm \rho}_{\bm m}(z) = \widetilde{\bf b}_{\bm m}(z), \quad {\rm A}_{\bm{m}}(z)=\frac{2}{z} \mathrm{D}^{1/2}_{\bm{m}}
 \left(
 \mathrm{I}_m + \tfrac{1}{2} z \mathrm{M}_{\bm{m}}
 \right) \mathrm{D}^{1/2}_{\bm{m}},
 \end{equation}
 where
 \begin{equation}
 \label{eq:47}
\widetilde{\bf b}_{\bm{m}}(z) := (2\gamma_{m_0+2}(z), 2\gamma_{m_0+3}(z), \ldots, 2\gamma_{m_1+1}(z))^\top.
\end{equation}
Thus, the contributions of $\rho_{m_0}(z)$ and $\rho_{m_1+1}(z)$ in the first and last entries of the right-hand side vector ${\bf b}_{\bm m}(z)$ have been removed.

The entry $ \widetilde{\rho}_{m_{\rm c}}(z) $ of the vector
\[
 \widetilde{\bm \rho}_{\bm m}(z) = \big( \widetilde{\rho}_{m_0+1}(z),\widetilde{\rho}_{m_0+2}(z),\dots,\widetilde{\rho}_{m_1}(z) \big)^\top
\]
which solves the linear system \eqref{eq:46}, is the quantity of interest in the third algorithm and corresponds to the middle entry of this vector.
The key idea is that $ \widetilde{\rho}_{m_{\rm c}} \approx \rho_{m_{\rm c}} $ within machine precision, provided that $ \bm{m} $ is appropriately chosen. For this {analysis}, we require the following lemma, which provides an estimate of the magnitude of the coefficient(s) being computed.
\begin{lemma}
For $\sigma =\Re z \leq 0$, the following bound holds:
 \begin{equation}\label{eq:bmtilde}
  |\rho_{n}(z)|\leq { \min\left\{2 , \left[\frac{1}{2|\sigma|}\left(\log(n+1)+2\right)^{1/2}\right]^{1/2}  \right\}}.
 \end{equation}
\end{lemma}\begin{proof}
{
We begin by noting that
\[
\begin{aligned}
 |\rho_n(z)| \ \leq\ & \frac{1}{n+1}\int_{0}^{2} |T'_{n+1}(s-1)|\,{\rm d}s
  = \int_{0}^\pi |\sin((n+1)\theta)|\,{\rm d}\theta\\
 =\ & \sum_{j=0}^{n}\int_{j\pi/(n+1)}^{(j+1)\pi/(n+1)} |\sin((n+1)\theta)|\,{\rm d}\theta = 2.
\end{aligned}
\]
On the other hand,
\begin{equation}\label{eq:LUX:byRosalia:02}
\begin{aligned}
 |\rho_n(z)| \ \leq\ &
 \frac{1}{n+1}\left[\int_{-1}^{1} |T'_{n+1}(x)|^{2}\,{\rm d}x\right]^{1/2}
 \left[\int_{0}^{2} |e^{2zs}|\,{\rm d}s\right]^{1/2} \\
 &\leq |2\sigma|^{-1/2}
 \left[\int_{0}^\pi \left|\frac{\sin((n+1)\theta)}{\sin \theta}\right|^{2}\sin\theta\,{\rm d}\theta\right].
\end{aligned}
\end{equation}
The identity (see \eqref{exp:Tj})
\[
 \left[\frac{\sin((n+1)\theta)}{\sin \theta}\right]^{2}
   = \left[\sum_{j=0}^{n} e^{{\rm i}(2j-n)\theta}\right]^{2}
   = (n+1) + 2\sum_{j=1}^{n} (n+1-j)\cos(2j\theta),
\]
yields
\[\label{eq:LUX:byRosalia}
\begin{aligned}
\int_{0}^\pi \left|\frac{\sin((n+1)\theta)}{\sin \theta}\right|^{2}\sin\theta\,{\rm d}\theta
 &= 2(n+1) + 4\sum_{j=1}^{n}\frac{n+1-j}{1-4j^{2}} \\
 &= 2\sum_{j=0}^{n}\frac{1}{2j+1}
  = 2\sum_{j=1}^{2n+1}\frac{1}{j} - \sum_{j=1}^{n}\frac{1}{j}\\
 &\leq \log\left(\frac{(2n+1)^{2}}{n}\right) + \gamma + \frac{1}{2n+2},
\end{aligned}
\]
where we have used \cite[Eq. (2.10.8)]{NIST:DLMF}, and where $\gamma \approx 0.5772$ denotes the Euler–Mascheroni constant.}

{As   a consequence, for $n\ge 8$, we obtain the following simpler bound, which is only slightly weaker:
\begin{equation}\label{eq:bound:Tn}
\int_{0}^\pi \left|\frac{\sin((n+1)\theta)}{\sin \theta}\right|^{2}\sin\theta\,{\rm d}\theta
 \leq \log(n+1)+2.
\end{equation}
For $n\le 7$, this bound also holds and can be verified directly by explicit computation.}

{Finally, combining \eqref{eq:LUX:byRosalia:02} and  \eqref{eq:bound:Tn} yields the last part of the proof.}
\end{proof}

\begin{proposition}
{Assume that $|z|\ge 1$} and let $r>0$.
Take $m_0, m_1 \in \mathbb{N}$ with $m_1 > m_0 \ge (1+r)|z| - 2$, and suppose that
$m = m_1 - m_0$ is odd.  Set $\bm{m}=(m_0,m_1)$ and let $\widetilde{\rho}_{\bm{m}}(z)$ be given by \eqref{eq:46}. Then,
\[
| \widetilde{\rho}_{m_{\rm c}}(z)- {\rho}_{m_{\rm c}}(z)|\leq
{3} r^{-1}(1+r)^{-(m+1)/2},
\]
where $m_{\rm c} :=\frac{m_1+m_0+1}{2}$.
\end{proposition}
\begin{proof}
{According to the most pessimistic bound given in the previous lemma and the definition~\eqref{eq:gamman}}, we can bound
\[
\| {\bf b}(z)\|_\infty\leq  \max\{|\rho_{m_0}(z)|,|\rho_{m_1+1}(z)|\} + \frac{1}{|z|}< {3}.
\]
By Proposition \ref{theo:stability:phase02},
\[
\left\|\frac{z}2\mathrm{M}_{\bm{m}}\right\|_{\infty}\leq \frac{|z|}{m_0+2}\le \frac{1}{1+r}<1,
\]
it follows that
\[
{\bm{\rho}}_m(z)= \frac{z}2 \mathrm{D}_{\bm {m}}^{-1/2}(z) \left[\sum_{j=0}^\infty \left(-\tfrac12 z \mathrm{M}_{\bm{m}}\right)^j\right]\mathrm{D}_{\bm {m}}^{-1/2}(z){\bf b}_m(z).
\]
We now define the finite-sum related vector:
\[
\check{\bm{\rho}}_m(z) = \frac{z}2
\mathrm{D}_{\bm {m}}^{-1/2}(z) \left[\sum_{j=0}^{(m-3)/2}\left(-\tfrac12 z \mathrm{M}_{\bm{m}}\right)^j\right]\mathrm{D}_{\bm {m}}^{-1/2}(z){\bf b}_m(z).
\]
Noting that
\[
{\bm{\rho}}_m(z)-\check{\bm{\rho}}_m (z)= \frac{z}2 \mathrm{D}_{\bm {m}}^{-1/2}(z) \left[\sum_{j=(m-1)/2}^\infty \left(-\tfrac12 z \mathrm{M}_{\bm{m}}\right)^j\right]\mathrm{D}_{\bm {m}}^{-1/2}(z){\bf b}_m(z)
\]
we obtain
\begin{eqnarray*}
 \| {\bm{\rho}}_m(z)-\check{\bm{\rho}}_m (z)\|_{\infty} &\leq&
 \frac{|z|}2 \|\mathrm{D}_{\bm {m}}^{-1/2}(z)\|_{\infty} \left[\sum_{j=(m-1)/2}^\infty \frac{|z|^{j}}{2} \|\mathrm{M}_{\bm{m}}\|_{\infty}^j\right] \|\mathrm{D}_{\bm {m}}^{-1/2}(z)\|_{\infty} \|{\bf b}_m(z)\|_{\infty}\\
 &\leq&
 \frac{1}2\left[\frac{|z|}{ m_0+2 }\right]^{ (m+1)/2} \left[1-\frac{|z|}{m_0+2}\right]^{-1} \frac{5}{|z|}\\
  &\leq & \frac{5}{2|z|} r^{-1}(1+r)^{-(m+1)/2}.
\end{eqnarray*}
where the following bounds have been used:
\[
\|\mathrm{D}_{\bm {m}}^{-1/2}(z)\|_{\infty}\le \frac{1}{\sqrt{m_0+2}},\quad \frac{|z|}{m_0+2}\leq \frac{1}{1+r}.
\]
Finally, since $ \mathrm{M}_{\bm{m}}(z) $ is tridiagonal, the bandwidth of $ \mathrm{M}_{\bm{m}}^j $ is just $ j+1 $ , which implies that all entries of
\[
 \mathrm{M}_{\bm{m}}^j(\widetilde{\bf b}_{\bm m}(z)-{\bf b}_{\bm m}(z))
\]
are zero except for the first and last $ j+1 $ entries. This ensures that
\[
 \check{\rho}_{m_{\rm c}}(z) = \widetilde{\rho}_{m_{\rm c}}(z).
\]
The result is then proven.
\end{proof}

\begin{remark}
Notice that in view of the previous result,  for $\varepsilon>0$ small, it suffices to take
\begin{equation}\label{eq:m}
 m= {2\left\lfloor\frac{\log(3/(\varepsilon  r))}{\log(1+r)} \right\rfloor+1}
\end{equation}
to ensure that
 \[
| \widetilde{\rho}_{m_{\rm c}}(z)- {\rho}_{m_{\rm c}}(z)|\leq \varepsilon.
\]
We can summarize these calculations in Algorithm \ref{alg:03}.
\end{remark}

\begin{algorithm}
\begin{algorithmic}[1]
    \Require $z$, $L$, $r>0$, $\varepsilon>0$.
    \State Set  $m_0= \max\{\lfloor (1+r)|z|\rfloor-2,L\}$.
    \State Choose an odd integer $m$ satisfying \eqref{eq:m}, and define $m_1 := m_0 + m$.
    \State Set $ \bm{m} := (m_0, m_1) $.
    \State Compute $ \widetilde{\bf b}_{\bm{m}}(z)$, $\mathrm{D}_{\bm{m}}$, and $\mathrm{M}_{\bm{m}}$ according to \eqref{eq:47}, \eqref{eq:Dn}, and \eqref{eq:Mn}.
    \State Set ${\bf c}_{\bm{m}} = \frac{z}{2} \mathrm{D}_{\bm{m}}^{-1/2} \widetilde {\bf b}_{\bm{m}}(z)$.
    \State Solve $(\mathrm{I}_m + \frac{z}2 \mathrm{M}_{\bm{m}}) {\bm y}_{\bm{m}} = {\bf c}_{\bm{m}}$.
 %\label{eq:tridiagonalsystem}.
    \State Compute $\widetilde{\bm{\rho}}_{\bm{m}}(z) =   \mathrm{D}_{\bm{m}}^{-1/2} {\bm y}_{\bm{m}}$.
    \State \Return  $n_1=(m_1+m_0-1)/2$, $ \widetilde{\rho}_{n_1+1}(z)$
\end{algorithmic}
\caption{\label{alg:03}Computation of $n_1$ and $\widetilde{\rho}_{n_1+1}(z)\approx \rho_{n_1+1}(z)$.}
\end{algorithm}

\subsection{Summary: Fast and stable computation of the weights for all $L$ and $z$}
We summarize the three algorithms introduced in Algorithm \ref{alg:04}:  Given inputs $L$ and $z$, the algorithm returns the vector of weights $\bm{\omega}_L(z)=(\omega_{0}(z), \ldots, \omega_{L}(z))^\top$  (see \eqref{eq:thequadformula}) as well as $\bm{\rho}_L(z)=(\rho_{0}(z), \ldots, \rho_{L}(z))^\top$ which may also be of interest.

\begin{algorithm}
\begin{algorithmic}[1]
    \Require $z$, $L$
    \If{$\Re z=0$}
    \State Set $n_0=\lceil |z|\rceil +1$
    \Else
    \State Set $n_0=2\lceil \ |z|^{1/2} \rceil +1$
    \EndIf
    \State Set $L_0=\min\{n_0,L\}$
    \State{Compute $(\rho_{0}(z), \ldots, \rho_{L_0}(z))$ and $(\omega_{0}(z), \ldots, \omega_{L_0}(z))$
    \Statex using Algorithm \ref{alg:01} with inputs $L_0$ and $z$.}
    \If{$L_0<L$}
    \State Compute $n_1$ and ${\rho}_{n_1+1}(z)$ using Algorithm \ref{alg:03} with inputs $L$, $z$, $r{>0}$, and $\varepsilon$.
    \State Compute $(\rho_{n_0+1}(z), \ldots, \rho_{L}(z))$ and $(\omega_{n_0+2}(z), \ldots, \omega_{L}(z))$
    \Statex\hspace{\algorithmicindent}using Algorithm \ref{alg:02} with $L$, $z$, $n_0$, and ${\rho}_{n_0+1}(z)$,   $n_1$ and ${\rho}_{n_1+1}(z)$ as inputs.
    \EndIf
    \State \Return ${\bm{\rho}_L(z):=}(\rho_{0}(z), \ldots, \rho_{L}(z))$ and ${\bm{\omega}_L(z):=}(\omega_{0}(z), \ldots, \omega_{L}(z))$.
\end{algorithmic}
\caption{\label{alg:04}Fast and stable computation of the weight vectors $\bm{\omega}_{L}(z)$ and $\bm{\rho}_{L}(z)$.}
\end{algorithm}

\section{Numerical experiments}

We illustrate the main results of this paper with selected numerical experiments. Hence, we first explore the stability of the evaluation of the weights $\{\bm{\omega}_L(z),\bm{\rho}_L(z)\}$ on which the efficient implementation of the quadrature method hinges.

Next, we test the convergence of the quadrature rule in several cases, for both smooth and non-smooth functions, in terms of the number of quadrature points and $z$, the (complex) phase in the exponential.

Finally, we introduce an application of these quadrature rules for solving fractional partial evolutionary problems.
%{All experiments were carried out using MATLAB on a Linux system equipped with an Intel Xeon E5-2630 v3 processor.}

\subsection{Implementation of the rule}

\subsubsection{Stability in the computation of the weights}

We have implemented the algorithm for evaluating the weights, which consists of Phase 1 and Phase 2 as outlined in Algorithms \ref{alg:01} and \ref{alg:02}, using MATLAB. Additionally, a quadruple-precision implementation was developed using MATLAB’s symbolic toolbox. While this version is more robust, it is significantly less efficient and will be used only for testing purposes. The code can be found in \cite{CleCurRules}.

Let us then denote
\[
 \bm{\rho}^{\rm unstab}_n(z), \
 \bm{\rho}^{\rm unstab,qp}_n(z),\
 \bm{\rho}_n(z),\
 \bm{\rho}^{\rm qp}_n(z),
\]
where $
 \bm{\rho}^{\rm unstab}_n(z)$ and $ \bm{\rho}^{\rm unstab,qp}_n(z)$ are the results of applying only Phase 1 of Algorithm \ref{fig:01} in double and quadruple precision (in other words, $n_0=\infty$ in Algorithm \ref{alg:04}). On the other hand, $\bm{\rho}_n(z)$ and $\bm{\rho}^{\rm qp}_n(z)$ are the results obtained from the stable algorithm (such as presented in Algorithm \ref{alg:04}).

 Figure \ref{fig:01} shows the modulus of $\bm{\rho}^{\rm unstab}_{256}(z)$ and $\bm{\rho}^{\rm unstab,qp}_{256}(z)$ for $z=-40\pi e^{{\rm i}\theta}$ with $\theta = 0, \pi/6, \pi/3$, and $\pi/2$.
We observe that instabilities clearly affect the computations, with the worst case occurring at $z = -40\pi \approx -125$, as conjectured. Indeed, for purely imaginary numbers, the recurrence relation in Algorithm 1 remains stable up to $n_0\approx |z|$, a fact that, as mentioned in Section 3, was proven analytically in \cite{DoGrSm:2010}. We conclude also that quadruple-precision arithmetic only mitigates the problem.

In Figure \ref{fig:02}, we compare $
 \bm{\rho}^{\rm unstab,qp}_{256}(z)$ and
 $\bm{\rho}_{256}(z)$. We point out that the switch between Algorithms 1 and 2 occurs in this case at $n_0=26$, but this does not affect the quality of the approximation provided by Algorithm \ref{alg:04}. The error that appears far beyond this switching point is attributed entirely to the instabilities in the first phase of the algorithm, which ultimately also affect $
 \bm{\rho}^{\rm unstab,qp}_n(z)$.

% \pgfplotsset{
%     every axis plot/.append style={thick},
%     legend style={font=\scriptsize},
%     tick label style={font=\small}, % Corrected tick label style
%     title style={font=\small},
%     scale=0.5,
%     yminorticks=true,
%     yticklabel pos=left, % Ensure only one yticklabel position is used
%     yticklabel style={anchor=east}, % Aligns y-tick labels
%     label style={font=\small}, % Makes axis labels smaller
%     legend pos=north west, % Uncommented to enable legend positioning
% }
 \begin{figure}
%\centering{\input{fig01.tex}}
\centering{\includegraphics[width=0.66\textwidth]{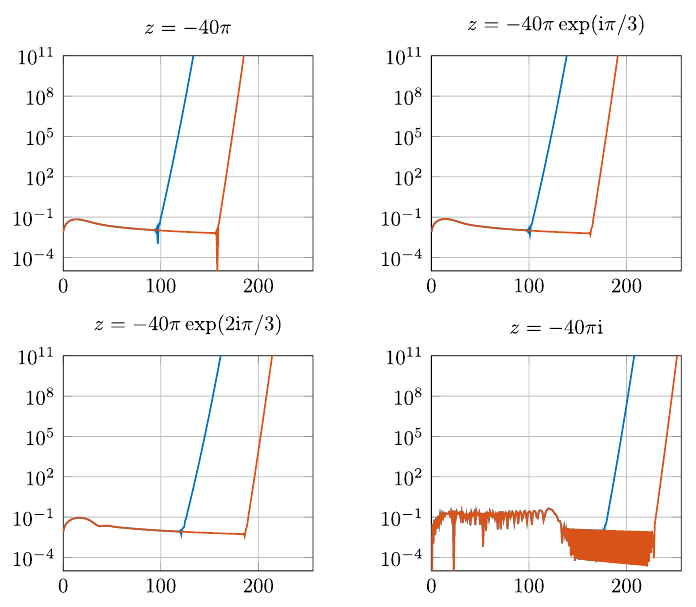}}

  \caption{\label{fig:01}  Modulus of the weights $ \bm{\rho}^{\rm unstab}_n(z),$ (blue) and
 $\bm{\rho}^{\rm unstab,qp}_n(z),$ (orange) for $n=0,\ldots,250$ and different values of $z$.  We observe that instabilities affect the computation, the worst scenario occurs when $z =-40\pi\approx 126$. Quadruple precision arithmetic only  mitigates the problem }

\centering{\includegraphics[width=0.66\textwidth]{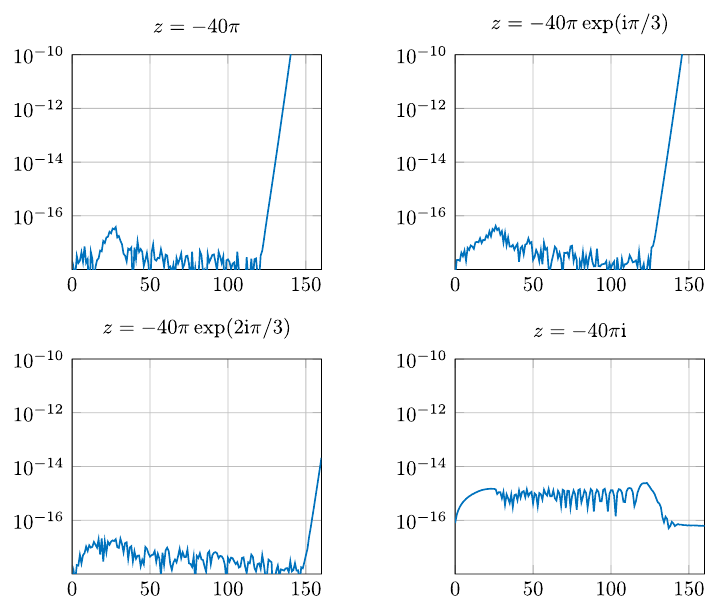}}
 %\centering{\input{fig04.tex}}
  \caption{\label{fig:02} Error $|\bm{\rho}_n(z)- \bm{\rho}^{\rm unstab,qp}_n(z)|$ for several values of $z$. We observe that $\bm{\rho}_n(z)$ is computed stably, even when Algorithm 02 is applied from $n_0=26$ onward. The error for larger values of $n$ in the three first cases can be attributed to the instabilities of Algorithm 01, which ultimately also affect the values computed in quadruple precision, $\bm{\rho}^{\rm unstab,qp}_n(z)$.}
\end{figure}
Finally, in Figure \ref{fig:03}, we display the modulus of the entries of vector $ \bm{\rho}_{256}(z)- \bm{\rho}^{\rm qp}_{256}(z)$. We point out that the computation of $\rho_{n_1}$ for some $n_1(z)$ was carried out using Algorithm 3 in both cases. However, for the quadruple-precision implementation, $n_1$ was chosen accordingly to improve the quality of the approximation, taking advantage of the extra precision in the calculations. The error remains within machine precision in double precision calculations.

\begin{figure}
%\centering{\input{fig03.tex}}%
\centering{\includegraphics[width=0.66\textwidth]{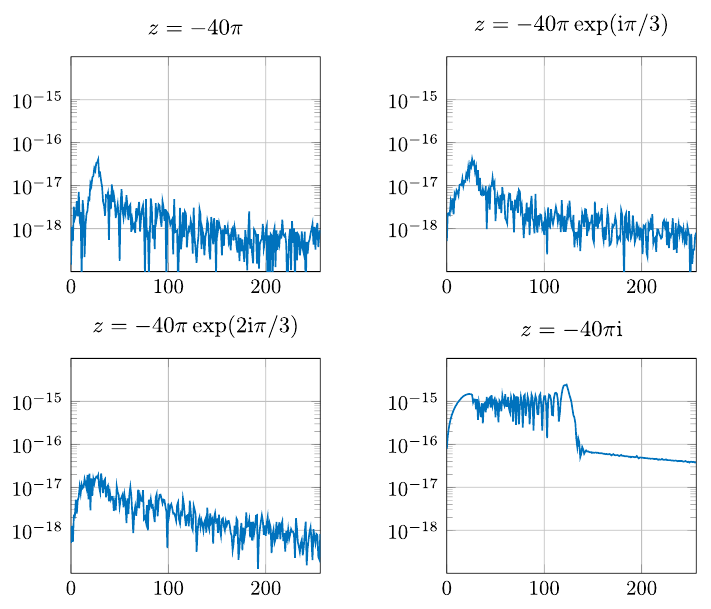}}
  \caption{\label{fig:03} Error between $|\bm{\rho}_n(z)-
 \bm{\rho}^{\rm qp}_n(z)|$ for several values of $z$. The calculations remains stable in all the range}
\end{figure}

\subsubsection{Error of the quadrature rule}

Consider the integral in this experiment:

\begin{equation}\label{eq:Inz}
I(n,z)= \int_{0}^{2} P_n(x-1) \exp(s z)\,{\rm d}s = \sqrt{\frac{2 \pi}{-{\rm i}z}} (-{\rm i})^n e^z J_{n+\frac{1}{2}}(-{\rm i}z), \quad z \in \mathbb{C} \setminus (-\infty,0]{\rm i}.
\end{equation}
Here, $P_n$ denotes the Lagrange polynomial of degree $n$.

This identity is given in equation (7.243) of \cite{GrRy:2007} (see also \cite{Bakhvalov1968}) for $z$ in the range
$[0,\infty){\rm i}$. The extension to other values of $z$ follows by analyticity.  We note that this formula has been previously considered in the literature for computing oscillatory integrals, particularly in the pioneering work of \cite{Bakhvalov1968}.

Computing such integrals allows us to test both the stability in the computation of the weights $\bm{\omega}_L(z)$ and the interpolation procedure behind the rule. We have applied our rule with $ L = n $. In exact arithmetic, this would yield the exact value.
As shown in Figure~\ref{fig:04}, the error remains on the order of the round-off unit.  However, the result is significantly worse when analyzing the relative errors, as illustrated in Figure~\ref{fig:05}. This discrepancy can be easily explained by the magnitude of the involved coefficients.

% \pgfplotsset{
%     every axis plot/.append style={thick},
%     legend style={font=\scriptsize},
%     tick label style={font=\small}, % Corrected tick label style
%     title style={font=\small},
%     scale=0.5,
%     scaled y ticks=false,
%     yminorticks=true,
%     yticklabel pos=left, % Ensure only one yticklabel position is used
%     yticklabel style={anchor=east}, % Aligns y-tick labels
%     label style={font=\small}, % Makes axis labels smaller
%     legend pos=north west, % Uncommented to enable legend positioning
% }
\begin{figure}
%\centering{ \input{figExp2_1.tex}}
\centering{\includegraphics[width=0.98\textwidth]{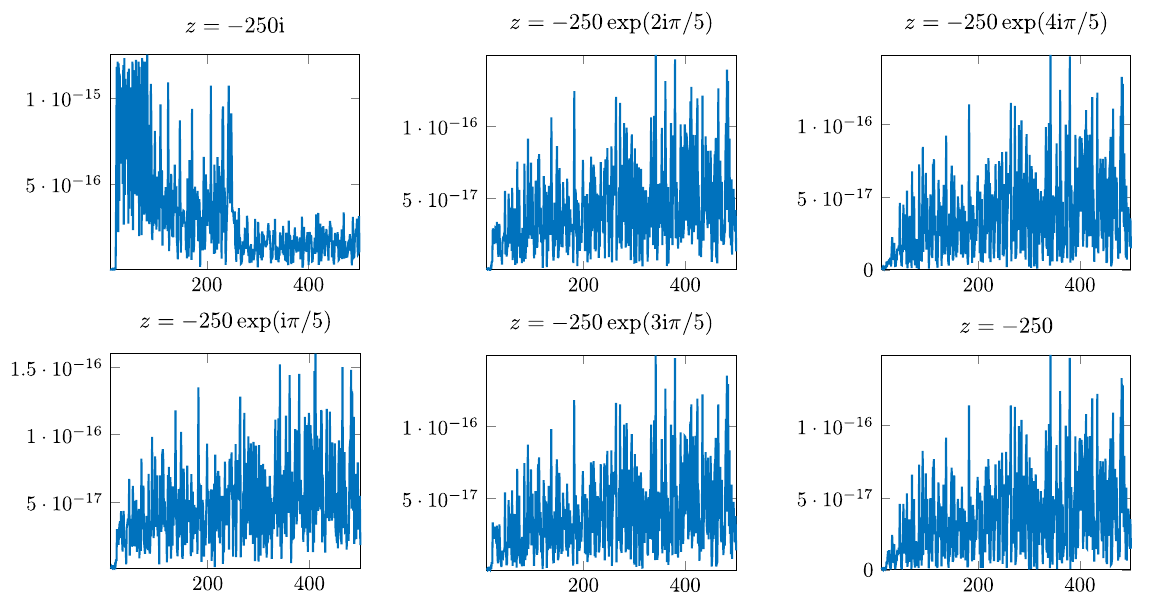}}
\caption{\label{fig:04}Absolute error in the computation of $I(n,z)$ cf. \eqref{eq:Inz} for $|z|=250$. Notice that the error in exact arithmetic should be zero }
\end{figure}

For instance, consider the following illustration:
\[
 I(128,-250)\approx 3.37\times 10^{-18}
\]
whereas
\[
\|\bm{\omega}_{129}(-250)\|_\infty = 4\times 10^{-3},
\]
which corresponds to a difference of 15 orders of magnitude.

\begin{figure}[h]
%\centering{ \input{figExp2_2.tex}}
\centering{\includegraphics[width=1\textwidth]{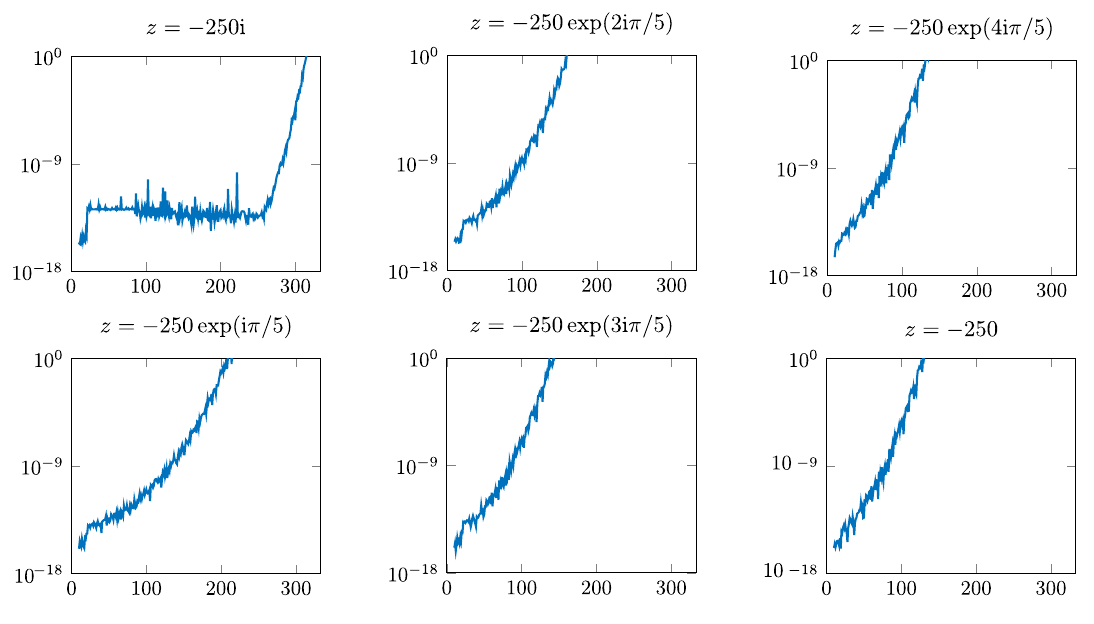}}
\caption{\label{fig:05}Relative error in the computation of  $I(n,z)$ cf. \eqref{eq:Inz}  for $|z|=250$. In this case the exact integral decay exponentially to zero which affects the computations due to cancellation error in the computation of the interpolating polynomial in the Chebyshev basis. }
\end{figure}
%
%
% \begin{figure}
%  \input{figExp2_3.tex}
% \caption{\label{fig:exp23} Even elements of the vector $\bm{\omega}_{128}(-250)$ used to compute $I_{128}(-250) \approx 3.37 \times 10^{-18}$. This calculation results in significant cancellation errors, which explain why the relative error of the quadrature rule is large, despite the relative error remaining within the round-off unit.}
% \end{figure}

Next we consider an academical example:
\begin{equation}\label{eq:LUX:byRosalia:03}
 J(z) = \int_{0}^2 \frac{\cos(5\pi s)}{4+\sin(4\pi s)}\exp({zs})\,{\rm d}s
\end{equation}
The error of the formula is expected to decay superalgebraically with respect $L$ in this case since the slow part of the integrand is smooth, actually analytic.

In Table \ref{tab:04}, we present the estimated errors for $ z = z_0\times 4^r $, with $ r=0,1,\ldots,5 $ and $ z_0=-20\exp(\pi \ell {\rm i}/6) $, with $ \ell =0,1,2,3 $. For reference, we have taken as the integral's exact value   the one computed using our quadrature rule with $L=1,280$ nodes.
We clearly observe superalgebraic convergence in $ L $ until machine precision is reached.
Additionally, when reading the tables row-wise, we observe that in the first rows, the error follows an approximate decay rate of $|z|^{-2}$, as predicted by Theorem \ref{theo:CleCur}.

\begin{table}[p]
\small
\[
\begin{array}{|r|c|c|c|c|c|c|}
\hline
L  \backslash r & 0 & 1 &2 &3 & 4 &5\\ \hline
 10   & 1.66{\rm e-}{04}  &   1.91{\rm e-}{04}  &   2.31{\rm e-}{05}  &   1.68{\rm e-}{06}  &   1.09{\rm e-}{07}  &   6.90{\rm e-}{09} \\
 20   & 1.88{\rm e-}{07}  &   1.39{\rm e-}{07}  &   1.76{\rm e-}{07}  &   2.12{\rm e-}{08}  &   1.54{\rm e-}{09}  &   1.00{\rm e-}{10} \\
 40   & 4.27{\rm e-}{08}  &   6.08{\rm e-}{08}  &   1.25{\rm e-}{08}  &   2.56{\rm e-}{08}  &   3.12{\rm e-}{09}  &   2.28{\rm e-}{10} \\
 80   & 2.97{\rm e-}{14}  &   3.17{\rm e-}{14}  &   4.60{\rm e-}{14}  &   7.39{\rm e-}{15}  &   1.85{\rm e-}{14}  &   2.25{\rm e-}{15} \\
160   & 2.60{\rm e-}{18}  &   4.34{\rm e-}{19}  &   6.51{\rm e-}{19}  &   1.08{\rm e-}{18}  &   2.03{\rm e-}{19}  &   4.34{\rm e-}{19} \\
320   & 1.73{\rm e-}{18}  &   0.00E+00  &   1.08{\rm e-}{19}  &   2.71{\rm e-}{20}  &   1.36{\rm e-}{20}  &   3.39{\rm e-}{21} \\
640   & 1.73{\rm e-}{18}  &   0.00E+00  &   1.08{\rm e-}{19}  &   2.71{\rm e-}{20}  &   0.00E+00  &   1.69{\rm e-}{21} \\
\hline \end{array}\]
%\caption{\label{tab:01}Estimated error of the quadrature rule {for integral \eqref{eq:LUX:byRosalia:02}}  with $L+1$ points  for $z =-20\times 4^r$}
\[
\begin{array}{|r|c|c|c|c|c|c|}
\hline
L  \backslash r & 0 & 1 &2 &3 & 4 &5\\ \hline
 10   & 6.73{\rm e-}{04}  &   2.21{\rm e-}{04}  &   2.38{\rm e-}{05}  &   1.70{\rm e-}{06}  &   1.10{\rm e-}{07}  &   6.90{\rm e-}{09} \\
 20   & 1.91{\rm e-}{06}  &   6.29{\rm e-}{07}  &   2.03{\rm e-}{07}  &   2.18{\rm e-}{08}  &   1.55{\rm e-}{09}  &   1.00{\rm e-}{10} \\
 40   & 4.18{\rm e-}{08}  &   5.07{\rm e-}{08}  &   8.20{\rm e-}{08}  &   2.94{\rm e-}{08}  &   3.21{\rm e-}{09}  &   2.30{\rm e-}{10} \\
 80   & 2.96{\rm e-}{14}  &   3.12{\rm e-}{14}  &   4.01{\rm e-}{14}  &   5.71{\rm e-}{14}  &   2.12{\rm e-}{14}  &   2.32{\rm e-}{15} \\
160   & 1.79{\rm e-}{18}  &   1.37{\rm e-}{18}  &   7.67{\rm e-}{19}  &   8.31{\rm e-}{19}  &   1.37{\rm e-}{18}  &   4.94{\rm e-}{19} \\
320   & 8.67{\rm e-}{19}  &   6.13{\rm e-}{19}  &   5.42{\rm e-}{20}  &   3.83{\rm e-}{20}  &   3.39{\rm e-}{21}  &   2.54{\rm e-}{21} \\
640   & 4.34{\rm e-}{19}  &   8.94{\rm e-}{19}  &   0.00E+00  &   2.71{\rm e-}{20}  &   6.78{\rm e-}{21}  &   1.89{\rm e-}{21} \\
\hline \end{array}\]
% \caption{\label{tab:02}Estimated error of the quadrature rule {for integral \eqref{eq:LUX:byRosalia:02}} with $L+1$ points  for $z =-20\times 4^r\exp(\pi{\rm i}/6)$}
 \[
\begin{array}{|r|c|c|c|c|c|c|}
\hline
L  \backslash r & 0 & 1 &2 &3 & 4 &5\\ \hline
 10   & 2.20{\rm e-}{03}  &   3.14{\rm e-}{04}  &   2.57{\rm e-}{05}  &   1.73{\rm e-}{06}  &   1.10{\rm e-}{07}  &   6.91{\rm e-}{09} \\
 20   & 8.32{\rm e-}{05}  &   2.57{\rm e-}{06}  &   2.89{\rm e-}{07}  &   2.36{\rm e-}{08}  &   1.58{\rm e-}{09}  &   1.01{\rm e-}{10} \\
 40   & 2.56{\rm e-}{08}  &   2.37{\rm e-}{07}  &   3.30{\rm e-}{07}  &   4.16{\rm e-}{08}  &   3.47{\rm e-}{09}  &   2.34{\rm e-}{10} \\
 80   & 1.95{\rm e-}{14}  &   3.01{\rm e-}{14}  &   1.18{\rm e-}{13}  &   2.25{\rm e-}{13}  &   3.00{\rm e-}{14}  &   2.51{\rm e-}{15} \\
160   & 3.13{\rm e-}{18}  &   6.13{\rm e-}{19}  &   6.79{\rm e-}{19}  &   4.52{\rm e-}{18}  &   5.52{\rm e-}{18}  &   7.00{\rm e-}{19} \\
320   & 3.47{\rm e-}{18}  &   8.94{\rm e-}{19}  &   3.64{\rm e-}{19}  &   3.03{\rm e-}{20}  &   7.58{\rm e-}{21}  &   4.24{\rm e-}{21} \\
640   & 2.45{\rm e-}{18}  &   8.94{\rm e-}{19}  &   2.24{\rm e-}{19}  &   4.89{\rm e-}{20}  &   6.78{\rm e-}{21}  &   4.24{\rm e-}{21} \\
\hline \end{array}\]
% \caption{\label{tab:03}Estimated error of the quadrature rule {for integral \eqref{eq:LUX:byRosalia:02}} with $L+1$ points  for $z =-20\times 4^r\exp(2\pi{\rm i}/6)$}
 \[
 \begin{array}{|r|c|c|c|c|c|c|}
\hline
L  \backslash r & 0 & 1 &2 &3 & 4 &5\\ \hline
 10   & 8.00{\rm e-}{03}  &   7.37{
m e-}{04}  &   6.43{\rm e-}{05}  &   8.33{\rm e-}{07}  &   1.84{\rm e-}{07}  &   1.61{\rm e-}{08} \\
 20   & 1.30{\rm e-}{02}  &   4.39{\rm e-}{05}  &   1.02{\rm e-}{06}  &   1.05{\rm e-}{07}  &   5.03{\rm e-}{09}  &   2.05{\rm e-}{10} \\
 40   & 2.77{\rm e-}{05}  &   1.18{\rm e-}{04}  &   2.68{\rm e-}{06}  &   8.20{\rm e-}{08}  &   7.01{\rm e-}{09}  &   5.24{\rm e-}{10} \\
 80   & 9.38{\rm e-}{11}  &   9.80{\rm e-}{07}  &   3.25{\rm e-}{09}  &   9.48{\rm e-}{11}  &   4.41{\rm e-}{12}  &   2.67{\rm e-}{13} \\
160   & 3.49{\rm e-}{17}  &   2.43{\rm e-}{17}  &   4.25{\rm e-}{14}  &   1.01{\rm e-}{15}  &   2.32{\rm e-}{17}  &   9.77{\rm e-}{19} \\
320   & 3.10{\rm e-}{17}  &   8.83{\rm e-}{17}  &   2.29{\rm e-}{17}  &   1.20{\rm e-}{17}  &   2.49{\rm e-}{18}  &   2.45{\rm e-}{19} \\
640   & 1.25{\rm e-}{17}  &   3.98{\rm e-}{17}  &   1.69{\rm e-}{17}  &   5.85{\rm e-}{18}  &   2.78{\rm e-}{18}  &   2.57{\rm e-}{19} \\
\hline \end{array}\]

  \caption{{ Estimated error of the quadrature rule for  \eqref{eq:LUX:byRosalia:03} using $L+1$ points.
The first table corresponds to $z = -40 \times 4^r$,
the second to $z = -40 \times 4^r \exp(\pi {\rm i}/6)$,
the third to $z = -40 \times 4^r \exp(\pi {\rm i}/3)$,
and the last one to $z = -40 \times 4^r {\rm i}$.}\label{tab:04}}
\end{table}

\begin{table}
\small
\[
\begin{array}{|r|c|c|c|c|c| }
\hline
L  \backslash r & 0 & 1 &2 &3 & 4  \\ \hline
        80 & 8.66{\rm e- }07 & 9.22{\rm e- }07 & 1.31{\rm e- }06 & 1.69{\rm e- }06 & 5.90{\rm e- }07 \\
       160 & 1.07{\rm e- }07 & 1.08{\rm e- }07 & 1.15{\rm e- }07 & 1.63{\rm e- }07 & 2.12{\rm e- }07 \\
       320 & 1.33{\rm e- }08 & 1.33{\rm e- }08 & 1.35{\rm e- }08 & 1.44{\rm e- }08 & 2.04{\rm e- }08 \\
       640 & 1.66{\rm e- }09 & 1.66{\rm e- }09 & 1.67{\rm e- }09 & 1.69{\rm e- }09 & 1.80{\rm e- }09 \\
      1280 & 2.07{\rm e- }10 & 2.07{\rm e- }10 & 2.07{\rm e- }10 & 2.08{\rm e- }10 & 2.11{\rm e- }10 \\
      2560 & 2.59{\rm e- }11 & 2.59{\rm e- }11 & 2.59{\rm e- }11 & 2.59{\rm e- }11 & 2.60{\rm e- }11 \\
      5120 & 3.24{\rm e- }12 & 3.24{\rm e- }12 & 3.24{\rm e- }12 & 3.24{\rm e- }12 & 3.24{\rm e- }12 \\
\hline \end{array}\]
%\caption{\label{tab:b:01}Estimated error of the quadrature rule for \eqref{eq:Ialpha}  with $\alpha=1/2$ with $L+1$ points  for $z =-40\times 4^r$}
\[
\begin{array}{|r|c|c|c|c|c| }
\hline
L  \backslash r & 0 & 1 &2 &3 & 4  \\ \hline
        80 & 8.65{\rm e- }07 & 9.15{\rm e- }07 & 1.28{\rm e- }06 & 1.80{\rm e- }06 & 6.04{\rm e- }07 \\
       160 & 1.07{\rm e- }07 & 1.08{\rm e- }07 & 1.14{\rm e- }07 & 1.60{\rm e- }07 & 2.25{\rm e- }07 \\
       320 & 1.33{\rm e- }08 & 1.33{\rm e- }08 & 1.35{\rm e- }08 & 1.43{\rm e- }08 & 2.00{\rm e- }08 \\
       640 & 1.66{\rm e- }09 & 1.66{\rm e- }09 & 1.67{\rm e- }09 & 1.69{\rm e- }09 & 1.79{\rm e- }09 \\
      1280 & 2.07{\rm e- }10 & 2.07{\rm e- }10 & 2.07{\rm e- }10 & 2.08{\rm e- }10 & 2.11{\rm e- }10 \\
      2560 & 2.59{\rm e- }11 & 2.59{\rm e- }11 & 2.59{\rm e- }11 & 2.59{\rm e- }11 & 2.60{\rm e- }11 \\
      5120 & 3.24{\rm e- }12 & 3.24{\rm e- }12 & 3.24{\rm e- }12 & 3.24{\rm e- }12 & 3.24{\rm e- }12 \\
\hline \end{array}\]
 %\caption{\label{tab:b:02}Estimated error of the quadrature rule for \eqref{eq:Ialpha}  with $\alpha=1/2$ with $L+1$ points   for $z =-40\times 4^r\exp(\pi{\rm i}/6)$}
 \[
\begin{array}{|r|c|c|c|c|c| }
\hline
L  \backslash r & 0 & 1 &2 &3 & 4  \\ \hline
        80 & 8.61{\rm e- }07 & 8.96{\rm e- }07 & 1.15{\rm e- }06 & 2.16{\rm e- }06 & 6.46{\rm e- }07 \\
       160 & 1.06{\rm e- }07 & 1.08{\rm e- }07 & 1.12{\rm e- }07 & 1.44{\rm e- }07 & 2.70{\rm e- }07 \\
       320 & 1.33{\rm e- }08 & 1.33{\rm e- }08 & 1.34{\rm e- }08 & 1.40{\rm e- }08 & 1.80{\rm e- }08 \\
       640 & 1.66{\rm e- }09 & 1.66{\rm e- }09 & 1.66{\rm e- }09 & 1.68{\rm e- }09 & 1.75{\rm e- }09 \\
      1280 & 2.07{\rm e- }10 & 2.07{\rm e- }10 & 2.07{\rm e- }10 & 2.08{\rm e- }10 & 2.10{\rm e- }10 \\
      2560 & 2.59{\rm e- }11 & 2.59{\rm e- }11 & 2.59{\rm e- }11 & 2.59{\rm e- }11 & 2.60{\rm e- }11 \\
      5120 & 3.24{\rm e- }12 & 3.24{\rm e- }12 & 3.24{\rm e- }12 & 3.24{\rm e- }12 & 3.24{\rm e- }12 \\
\hline \end{array}\]
 %\caption{\label{tab:b:03}Estimated error of the quadrature rule for \eqref{eq:Ialpha}  with $\alpha=1/2$ with $L+1$ points for $z =-40\times 4^r\exp(2\pi{\rm i}/6)$}
 \[
 \begin{array}{|r|c|c|c|c|c| }
\hline
L  \backslash r & 0 & 1 &2 &3 & 4 \\ \hline
         80 & 1.16{\rm e- }06 & 2.71{\rm e- }05 & 3.39{\rm e- }06 & 2.46{\rm e- }06 & 1.59{\rm e- }06 \\
        160 & 1.42{\rm e- }07 & 2.87{\rm e- }06 & 2.25{\rm e- }06 & 1.65{\rm e- }06 & 8.65{\rm e- }07 \\
        320 & 1.77{\rm e- }08 & 2.59{\rm e- }08 & 1.40{\rm e- }06 & 2.02{\rm e- }08 & 1.85{\rm e- }07 \\
        640 & 2.21{\rm e- }09 & 3.24{\rm e- }09 & 6.43{\rm e- }08 & 1.18{\rm e- }07 & 8.87{\rm e- }08 \\
       1280 & 2.76{\rm e- }10 & 4.04{\rm e- }10 & 2.63{\rm e- }10 & 2.73{\rm e- }08 & 1.72{\rm e- }08 \\
       2560 & 3.46{\rm e- }11 & 5.05{\rm e- }11 & 3.28{\rm e- }11 & 1.70{\rm e- }09 & 5.26{\rm e- }09 \\
       5120 & 4.32{\rm e- }12 & 6.32{\rm e- }12 & 4.10{\rm e- }12 & 5.97{\rm e- }12 & 4.79{\rm e- }10 \\
\hline \end{array}\]
\caption{Estimated error of the quadrature rule for \eqref{eq:Ialpha} with $\alpha=1/2$ using $L+1$ points.
The first table corresponds to $z =-40\times 4^r{\rm i}$,
the second to $z =-40\times 4^r \exp(\pi {\rm i}/6)$,
the third to $z =-40\times 4^r \exp(\pi {\rm i}/3)$,
and the last one to $z =-40\times 4^r {\rm i}$.}
\label{tab:b:04}
\end{table}

Next, we analyze how the convergence of the quadrature rule is affected by functions with singularities. As a first example, we consider the integral (cf. \cite[Eq. 3.387]{GrRy:2007}):
\begin{equation}\label{eq:Ialpha}
 \int_0^2 (2(2-s))^\alpha \exp(z s)\,{\rm d}s = (2\alpha)!! \pi \exp(z) z ^{-1/2-\alpha}I_{\alpha+1/2}(z),
\end{equation}
for $\alpha = -1/2, 1/2, 3/2, \dots$. In our {analysis}, we focus on the cases $\alpha = 1/2$ and $\alpha = 3/2$.

We anticipated that the convergence of the method, both in terms of $L$ and the magnitude of $z$, will exhibit some deterioration because the limited regularity of the integrand at the end points of the integration interval. However, we expect this deterioration in $L$ to be less severe than predicted by theory (Theorem \ref{theo:CleCur}), since singularities located at the endpoints of the integration interval tend to be less problematic than those in the interior when Chebyshev polynomial techniques are used.
\added[id=R2]{To gain insight into this phenomenon, we note that $f_{\alpha}(s)=\left(2(2-s)\right)^\alpha$, for positive non-integer values of $\alpha$, belongs to $H^{\alpha+1/2-\varepsilon}(0,2)$ for any $\varepsilon>0$. However, $(f_{\alpha})_{\#}(\theta)=f_\alpha(1+\cos(\theta))$ belongs to $H_{\#}^{2\alpha+1/2-\varepsilon}$.}
Indeed, in the results we obtained, which are collected in Tables \ref{tab:b:04} (for $\alpha=1/2$) and \ref{tab:c:04} (for $\alpha=3/2$), the estimated order of convergence in $L$ appears to be $3$ for $\alpha=1/2$ and $5$ for $\alpha=3/2$
\added[id=R2]{which should be compared with the theoretical orders (almost) $3/2-\varepsilon$ and $7/2-\varepsilon$ predicted by our analysis (cf.\ Theorem~\ref{theo:CleCur}), indicating that the theoretical bounds are not sharp in these cases.}
However, we do not observe a decay in $z$, which can be explained by the fact that the negative powers of $z$ obtained via integration by parts cannot be freely applied here due to the singularities of the integrand at the endpoints.
\added[id=R2]{We can therefore conclude that the regularity assumptions made in Theorem~\ref{theo:CleCur} in order to obtain negative powers of $|z|$ in the convergence estimates are essentially sharp, or at least significantly sharper than those required to describe the convergence with respect to $L$.}
Observe also that this collection of experiments demonstrates the effectiveness of our algorithm in efficiently handling a large number of quadrature points.

\begin{table}[H]
\small
\[
\begin{array}{|r|c|c|c|c|c| }
\hline
L  \backslash r & 0 & 1 &2 &3 & 4  \\ \hline
        80 & 1.41{\rm e- }10 & 1.27{\rm e- }10 & 5.51{\rm e- }11 & 4.35{\rm e- }10 & 1.22{\rm e- }10 \\
       160 & 4.50{\rm e- }12 & 4.41{\rm e- }12 & 3.97{\rm e- }12 & 1.71{\rm e- }12 & 1.36{\rm e- }11 \\
       320 & 1.41{\rm e- }13 & 1.41{\rm e- }13 & 1.38{\rm e- }13 & 1.24{\rm e- }13 & 5.33{\rm e- }14 \\
       640 & 4.42{\rm e- }15 & 4.41{\rm e- }15 & 4.39{\rm e- }15 & 4.31{\rm e- }15 & 3.88{\rm e- }15 \\
      1280 & 1.37{\rm e- }16 & 1.38{\rm e- }16 & 1.38{\rm e- }16 & 1.37{\rm e- }16 & 1.35{\rm e- }16 \\
      2560 & 3.09{\rm e- }18 & 4.37{\rm e- }18 & 4.36{\rm e- }18 & 4.32{\rm e- }18 & 4.28{\rm e- }18 \\
      5120 & 1.30{\rm e- }18 & 2.17{\rm e- }19 & 1.44{\rm e- }19 & 1.38{\rm e- }19 & 1.33{\rm e- }19 \\
\hline \end{array}\]
%\caption{\label{tab:b:01}Estimated error of the quadrature rule for \eqref{eq:Ialpha}  with $\alpha=1/2$ with $L+1$ points  for $z =-40\times 4^r$}
\[
\begin{array}{|r|c|c|c|c|c| }
\hline
L  \backslash r & 0 & 1 &2 &3 & 4  \\ \hline
        80 & 1.42{\rm e- }10 & 1.32{\rm e- }10 & 2.18{\rm e- }10 & 5.17{\rm e- }10 & 1.29{\rm e- }10 \\
       160 & 4.50{\rm e- }12 & 4.43{\rm e- }12 & 4.12{\rm e- }12 & 6.82{\rm e- }12 & 1.62{\rm e- }11 \\
       320 & 1.41{\rm e- }13 & 1.41{\rm e- }13 & 1.38{\rm e- }13 & 1.29{\rm e- }13 & 2.13{\rm e- }13 \\
       640 & 4.42{\rm e- }15 & 4.41{\rm e- }15 & 4.40{\rm e- }15 & 4.33{\rm e- }15 & 4.03{\rm e- }15 \\
      1280 & 1.37{\rm e- }16 & 1.38{\rm e- }16 & 1.38{\rm e- }16 & 1.37{\rm e- }16 & 1.35{\rm e- }16 \\
      2560 & 3.58{\rm e- }18 & 4.15{\rm e- }18 & 4.37{\rm e- }18 & 4.31{\rm e- }18 & 4.29{\rm e- }18 \\
      5120 & 8.68{\rm e- }19 & 3.83{\rm e- }20 & 2.04{\rm e- }19 & 1.42{\rm e- }19 & 1.38{\rm e- }19 \\
\hline \end{array}\]
 %\caption{\label{tab:b:02}Estimated error of the quadrature rule for \eqref{eq:Ialpha}  with $\alpha=1/2$ with $L+1$ points   for $z =-40\times 4^r\exp(\pi{\rm i}/6)$}
 \[
\begin{array}{|r|c|c|c|c|c| }
\hline
L  \backslash r & 0 & 1 &2 &3 & 4  \\ \hline
        80 & 1.43{\rm e- }10 & 1.42{\rm e- }10 & 8.83{\rm e- }10 & 8.36{\rm e- }10 & 1.52{\rm e- }10 \\
       160 & 4.51{\rm e- }12 & 4.48{\rm e- }12 & 4.44{\rm e- }12 & 2.77{\rm e- }11 & 2.61{\rm e- }11 \\
       320 & 1.41{\rm e- }13 & 1.41{\rm e- }13 & 1.40{\rm e- }13 & 1.39{\rm e- }13 & 8.67{\rm e- }13 \\
       640 & 4.42{\rm e- }15 & 4.42{\rm e- }15 & 4.41{\rm e- }15 & 4.37{\rm e- }15 & 4.33{\rm e- }15 \\
      1280 & 1.38{\rm e- }16 & 1.38{\rm e- }16 & 1.38{\rm e- }16 & 1.38{\rm e- }16 & 1.37{\rm e- }16 \\
      2560 & 4.45{\rm e- }18 & 4.26{\rm e- }18 & 4.34{\rm e- }18 & 4.30{\rm e- }18 & 4.30{\rm e- }18 \\
      5120 & 8.13{\rm e- }19 & 2.61{\rm e- }19 & 1.75{\rm e- }19 & 1.26{\rm e- }19 & 1.39{\rm e- }19 \\
\hline \end{array}\]
 %\caption{\label{tab:b:03}Estimated error of the quadrature rule for \eqref{eq:Ialpha}  with $\alpha=1/2$ with $L+1$ points for $z =-40\times 4^r\exp(2\pi{\rm i}/6)$}
 \[
 \begin{array}{|r|c|c|c|c|c| }
\hline
L  \backslash r & 0 & 1 &2 &3 & 4 \\ \hline
        80 & 1.88{\rm e- }10 & 1.92{\rm e- }08 & 3.73{\rm e- }09 & 2.85{\rm e- }10 & 2.82{\rm e- }10 \\
       160 & 6.00{\rm e- }12 & 5.05{\rm e- }10 & 4.62{\rm e- }10 & 3.35{\rm e- }10 & 9.27{\rm e- }11 \\
       320 & 1.88{\rm e- }13 & 2.76{\rm e- }13 & 6.83{\rm e- }11 & 4.66{\rm e- }13 & 7.49{\rm e- }12 \\
       640 & 5.89{\rm e- }15 & 8.62{\rm e- }15 & 8.06{\rm e- }13 & 1.46{\rm e- }12 & 1.06{\rm e- }12 \\
      1280 & 1.69{\rm e- }16 & 2.76{\rm e- }16 & 1.93{\rm e- }16 & 8.37{\rm e- }14 & 5.17{\rm e- }14 \\
      2560 & 1.52{\rm e- }17 & 1.60{\rm e- }17 & 1.23{\rm e- }17 & 1.20{\rm e- }15 & 3.96{\rm e- }15 \\
      5120 & 1.80{\rm e- }17 & 1.02{\rm e- }17 & 5.40{\rm e- }18 & 3.39{\rm e- }18 & 8.92{\rm e- }17 \\
\hline \end{array}\]
\caption{Estimated error of the quadrature rule for \eqref{eq:Ialpha} with $\alpha=3/2$ using $L+1$ points.
The first table corresponds to $z =-40\times 4^r{\rm i}$,
the second to $z =-40\times 4^r \exp(\pi {\rm i}/6)$,
the third to $z =-40\times 4^r \exp(\pi {\rm i}/3)$,
and the last one to $z =-40\times 4^r {\rm i}$.}
\label{tab:c:04}
\end{table}

 \subsection{Applications to the solutions of fractional order evolution equations}

 In this section, we present an application of the product Clenshaw-Curtis rule analyzed in this paper to the solution of fractional-time evolution partial differential equations via Laplace transform techniques. These techniques have been widely applied in various contexts. For an introduction to this topic, as well as the foundation of the approach taken in this section, we refer the reader to \cite{MR2068831,MR2207270,MR2607556} and references therein.

As a model problem, we consider the following: let $\Omega$ be a polygonal domain in $\mathbb{R}^d$, and consider for $t\in[0,T]$ the following evolutionary problem:
\begin{equation}
 \label{eq:PDEequation}
 \partial_t u -\partial_t^{-\alpha}\Delta u=f(t),\quad u(0)=u_0
\end{equation}
Here, $\Delta: D(\Delta)\subset L^2(\Omega) \to L^2(\Omega)$ is the Laplace operator with homogeneous Dirichlet boundary conditions on $\partial \Omega$. Clearly, it is an unbounded operator with a well-defined compact inverse $\Delta^{-1}:L^2(\Omega)\to H^{1}_0(\Omega) \subset L^2(\Omega)$.

The fractional time derivative $\partial_t^{-\alpha}$ with
$\alpha\in(-1,1)$ is the {Riemann--Liouville} fractional operator,
defined as follows: for $\alpha=0$ we simply set $\partial_t^0u=u$,
while for $\alpha\neq0$ we define
\[
(\partial_t^{-\alpha}u)(t)=
\begin{cases}
 \displaystyle\partial_t\!\int_0^t \frac{(t-s)^\alpha}{\Gamma(1+\alpha)}u(s)\,{\rm d}s,
 & \alpha\in(-1,0),\\[2ex]
 \displaystyle\int_0^t \frac{(t-s)^{\alpha-1}}{\Gamma(\alpha)}u(s)\,{\rm d}s,
 & \alpha\in(0,1).
\end{cases}
\]
{Unlike the Caputo fractional derivative, the
Riemann--Liouville fractional derivative  satisfies the identity
$\widehat{\partial_t^{-\alpha}u}(z)=z^{-\alpha}\,\widehat{u}(z)$
for all $\alpha\in(-1,1)$, where $\widehat{u}$ denotes here, and in what follows, the Laplace
transform of $u$.}
Note also that for $\alpha=0$ we recover the classical heat equation.

Under these assumptions, we can apply the Laplace Transform and   write
\[
 \widehat{u}(z)= \widehat{\cal E}(z) (u_0+\widehat{f}(z))
\]
where
\[
\widehat{g}(z) := \int_0^\infty g(t)\exp(-z t)\,{\rm d}t,\qquad
 \widehat{{\cal E}}(z):=z^\alpha (z^{1+\alpha}\mathrm{I}-\Delta)^{-1}:L^2(\Omega)\to H_0^1(\Omega)\subset L^2(\Omega).
\]
Denote the subset in the complex plane
\[
\Sigma_{\varepsilon,\beta}:= \{w\in\mathbb{C} : |\mathop{\rm Arg} z|>\beta \pi,\ |z|>\varepsilon \}
\]
and notice that for any $\beta<\max\{(1+\alpha)^{-1},1\}$ and $\varepsilon$ sufficiently small  there exists $C_{\alpha,\varepsilon}$ such that
\begin{equation}\label{eq:Laplace:bound}
 \|{\cal E}(z)\|_{L^2(\Omega)  \to L^2(\Omega)}\leq \frac{C_\alpha}{1+|z|}, \quad \forall z\in \Sigma_{\varepsilon,\beta}.
\end{equation}
We can then invoke Bromwich integral first and deform the contour $\Re \:z=\sigma_0>0$ so that we obtain the representation formula
\begin{equation}\label{eq:LaplaceRepresentationFormula}
 u(t)=\frac{1}{2\pi{\rm i}}
 \int_{\Gamma} \exp(zt)\widehat{\cal E}(z) (u_0+\widehat{f}(z))\,{\rm d}z,\quad \forall t>0
\end{equation}
where $\Gamma\subset \Sigma_{\beta}\cup\{z\in\mathbb{C}\ : \ |z|>\varepsilon\}$  homotopic with the curve $\Re \:z=z_0>0$ oriented in the direction of the increasing imaginary part. Several choices for $\Gamma$ have been proposed in the literature. We follow the approach suggested by
L{\'o}pez-Fern{\'a}ndez and  Palencia in \cite{MR2091405}, the hyperbola parameterized:
\begin{equation}
 \begin{aligned}
 z(\xi)&:=\lambda(1-\sin(\delta-{\rm i}\xi))
 =\lambda\big[(1-\sin\delta\cosh\xi)+{\rm i}\cos\delta \sinh\xi\big].
 \\
 &=\lambda\big(1-\sin\delta \cosh\xi +{\rm i}\cos \delta \sinh \xi)
 , \quad \xi\in\mathbb{R},
 \end{aligned}
\end{equation}
where $\lambda$, $\delta$ are parameters taken to make \eqref{eq:Laplace:bound} hold for $z\in\Lambda$.

However,
the main drawback when following this approach is that the Laplace transform of $f$ is needed which very
often is a very restrictive assumption. However, the use of $\widehat{f}(z)$ can be circumvented proceeding in a different way: define
 \[
  g(z,t)= \exp(zt)u_0+\int_0^t \exp(z(t-s))f(s)\,{\rm d}s.
 \]
 so that
\begin{eqnarray}
 \label{eq:laplace:02}
u(t)&=& \frac{1}{2\pi{\rm i}}\int_\Gamma    \widehat{\cal E}(z) \widehat{g}(z,t) {\rm d}z, \quad t\in[0,T].
\end{eqnarray}

For  suitable $\delta\in(0,\pi/2)$ and $\lambda>0$,   $\Gamma$ is the left branch of hyperbola in the complex
plane which cuts the real axis at $\lambda(1-\sin\delta)$
with asymptotes $y = \pm (x-\lambda)\cot \delta $.

For simplicity, we denote
\[
 z_j = z(jk), \quad z_{j}' = z'(jk), \quad k=k(T)>0
\]
and define
\begin{eqnarray}\label{eq:UN}
 U_N(t)\!&:=&\!  \frac{k}{2\pi{\rm i}  }\sum_{j=-N}^N
 \widehat{\cal E}(z_j)g(z_j,t)z_j'\approx \frac{1}{2\pi}\int_{-\infty}^\infty  {\cal E}(z(\xi))\widehat{g}(z(\xi),t)z'(\xi)\,{\rm d}\xi \label{eq:defUn}
\end{eqnarray}
as an approximation of the true solution $u$. We point out that the contour integral in  \eqref{eq:laplace:02} is being approximated with the rectangular rule over the interval $[-kN ,kN]$. Since the parameterized integral decays exponentially to zero, this ensures rapid convergence of the rule, as proven in \cite{MR2091405}.  Observe also that, from an implementation point of view,  the terms in \eqref{eq:defUn} can be computed independently, which makes the method naturally parallelizable.

On the other hand, we still need to evaluate
\[
  g(z,t)
  = \exp(zt)u_0+\frac{t}{2} \int_0^2 \exp\left(\frac{z t}{2} s\right) \tilde{f}(s)\,{\rm d}s,\quad
\tilde{f}_t(s) := f(t(1-s/2)).
\]
It is precisely in this integral that we apply our quadrature rule, yielding the full discrete scheme
\begin{eqnarray}\label{eq:UN:2}
 U_{N,L,h} (t)\!&:=&\!  \frac{k}{2\pi{\rm i}  }\sum_{j=-N}^N
 \widehat{\cal E}_h(z_j)g_L(z_j,t)z_j', \quad g_L(z,t):= \exp(zt)u_0+\frac{t}{2} I_{L, (t/2)z}(\tilde{f}_t).
\end{eqnarray}
{Notice that $U_{N,L,h}$ relies, among other aspects, on the quadrature rule with complex exponential parameters $(t/2)z_j$ (cf.\ $I_{L,(t/2)z}(\tilde{f}_t)$ in the definition of $g_L(z,t)$). This involves working with a wide range of exponential parameter values that must be taken into account in the computations.}
This issue becomes even more pronounced when multiple snapshots of the solution are required, i.e., when computing the solution for different values of $t$.

The final component of this full discretization, and for which the suffix $h$ stands for, involves replacing $\widehat{\cal E}(z)$ with $\widehat{\cal E}_h(z)$, a finite element solver for the equation so that
\[
v_h=\widehat{\cal E}_h(z)f
\]
is the solution of
\[
 \left\{
 \begin{array}{ll}
 v_h\in P_h \cap H_0^{1}(\Omega)\\
 \displaystyle z^{1+\alpha}\int_\Omega {v_h} w_h +\int_{\Omega }\nabla v_h\nabla w_h = z^{-\alpha} \int_{\Omega}fw_h, & \forall w_h\in P_h \cap H_0^{1}(\Omega) \\
 \end{array}
 \right.
\]
with $P_h$ a classical finite element space of continuous piecewise polynomials.

We do not develop a detailed study of the convergence of the fully numerical method. Instead, we focus on implementing the method and demonstrating its convergence.
In the numerical experiment we consider we have taken $\alpha=1/2$. The domain, and subdomains, and the initial quadratic mesh for our FEM solver (consisting of 552 elements and 1,165 nodes) as well as initial condition $u_0$ and $f(x,y,z)$
\begin{equation}\label{eq:fpe:data}
 u_0(x,y) = \begin{cases}
               1, &(x,y)\in\Omega_1\cup\Omega_2\cup\Omega_3\\
              0 , &(x,y)\ \text{otherwise},
              \end{cases}\qquad
f(x,y,t)= \begin{cases}
              \sin(t) &(x,y)\in\Omega_1\\
              3.5\cos(3t) &(x,y)\in\Omega_2\\
              1.25\cos(2t) &(x,y)\in\Omega_3\\
              0 , &(x,y)\ \text{otherwise},
              \end{cases}
% @(x,y,dom) x*0.*(dom==1)+(1+x*0).*(dom==25)+(1+x*0).*(dom==35)+(1+x*0).*(dom==45);
%@(t,x,y,dom)  x*0.*(dom==1)+(sin(t)+x*0).*(dom==25)+... (3.5*cos(3*t)+x*0).*(dom==35)+(1.25*sin(t*2)+x*0).*(dom==45);
\end{equation}
 on different subdomains are shown in Figure \ref{fig:FEMmesh}. Quadratic Lagrange finite elements were used in our experiment.  The initial mesh was obtained with GMSH package \cite{gmsh}. Finer meshes were constructed using successively uniform (i.e. RGB refinements) of the initial grid.
\begin{figure}[h]
 \[
   \includegraphics[width=0.5\textwidth]{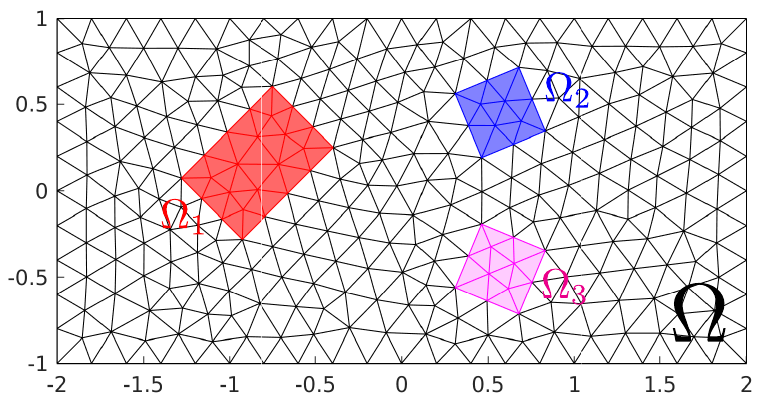}
 \]
\caption{\label{fig:FEMmesh}Domain of the problem}
\end{figure}
For $N \in \mathbb{N}$ we have taken
\[
 k = c_k \sqrt{N}, \quad c_k := \sqrt{\frac{4\pi r}{\gamma}} \approx 0.924,
\]
where, recall that $\alpha =0.5$,
\[
 r = \beta - 0.9(\delta + \pi/2),\quad  \beta = \frac{\pi}{4} + \frac{\pi}{2(1+\alpha)},\quad \delta = -\frac{\pi}{4} + \frac{\pi}{2(1+\alpha)},\quad \gamma = 1.8(1+\alpha).
\]
To tune the hyperbola used in our computations, we have set
\[
 \lambda = \frac{2\gamma}{\kappa T} \approx 5.86T, \quad \kappa = 1 - \sin(\delta - r) \approx 0.922.
\]
These values correspond to a parameter setting suggested in \cite{MR2607556},  to improve convergence in terms of $N$ and the order of the fractional derivative in the equation, a topic that we do not discuss further in this paper.

 Table \ref{tab:fpe} shows the estimated $L^2$ error for several time steps $N$ and different numbers of nodes $L$ in the quadrature rule. The spatial discretization used a mesh consisting of 35,328 triangles and 71,137 nodes. To analyze convergence in all the discrete parameters involved—namely $L$ and $N$—as well as the spatial  finite element discretization, we used as the {\em exact} solution the one obtained from the next refined mesh, which consists of 283,585 nodes and 141,312 triangles, with $N = 108$ and $L = 128$.

Although a detailed study of this equation is beyond the scope of this paper, several observations can be made at this first glance. First, for small values of $t$, even very low values of $L$ are sufficient to obtain highly accurate solutions. This can be attributed to the fact that the integral in $g_L(z_j,t)$ (see \eqref{eq:UN:2}) contributes negligibly or remains very small. However, as $t$ increases, the accuracy of the integral approximation becomes more critical. Nevertheless, even in the worst-case considered scenario, values as low as $L = 32$ yield sufficiently precise integral computations.

To gain insight into the values of the parameter $z$ involved in computing $g(z,t)$, we note that for $N = 52$ and the more demanding scenario $t = 4\pi$ in \eqref{eq:UN:2}, we have $(t/2)z_{\pm N} \approx -330 \pm 1250\mathrm{i}$ and $(t/2)z_0 \approx 2.17$. It is important to emphasize that the 105 integrals used in $U_{N,L,h}$ (cf. \eqref{eq:UN:2}), spanning such a wide range of values for $z$, are computed using our product Clenshaw-Curtis quadrature rules while maintaining the same configuration with $L = 32$ {\em nodes}, where $\tilde{f}_t$ is evaluated. Finally, we point out that the last case, i.e., $(t/2)z_0 \approx 2.17$, could be considered somewhat beyond the scope of the investigation in this paper, since its real part is positive although of moderate size. Nevertheless, the method still demonstrates resilience. We emphasize that 29 of these integrals correspond to $(t/2)z_j$ values with a positive real part

\begin{table}
 \[
\begin{array}{|crc|crc|crc|crc|}
     \hline
 \multicolumn{3}{|c|}{t=\pi/16} & \multicolumn{3}{|c|}{t=\pi/4} & \multicolumn{3}{|c|}{t=\pi} & \multicolumn{3}{|c|}{t=4\pi} \\
 \hline
 N & L & \text{Error} & N & L & \text{Error} & N & L & \text{Error} & N & L & \text{Error} \\
 \hline
 16 &  4 & 7.050{\rm e- }03 & 16 &  4  & 1.178{\rm e- }03 & 16 &  4 & 3.816{\rm e- }02  & 16 &  4 & 4.341{\rm e- }01 \\
    &  8 & 7.050{\rm e- }03 &    &  8  & 1.181{\rm e- }03 &    &  8 & 1.243{\rm e- }03  &    &  8 & 2.738{\rm e- }01 \\
    & 16 & 7.050{\rm e- }03 &    & 16  & 1.181{\rm e- }03 &    & 16 & 1.249{\rm e- }03  &    & 16 & 5.623{\rm e- }03 \\
    & 32 & 7.050{\rm e- }03 &    & 32  & 1.181{\rm e- }03 &    & 32 & 1.249{\rm e- }03  &    & 32 & 8.629{\rm e- }03 \\
    & 64 & 7.050{\rm e- }03 &    & 64  & 1.181{\rm e- }03 &    & 64 & 1.249{\rm e- }03  &    & 64 & 8.629{\rm e- }03 \\
 \hline
 24 &  4 & 7.684{\rm e- }04 & 24 &  4  & 1.575{\rm e- }04 & 24 &  4 & 3.795{\rm e- }02  & 24 &  4 & 4.403{\rm e- }01 \\
    &  8 & 7.684{\rm e- }04 &    &  8  & 1.658{\rm e- }04 &    &  8 & 2.133{\rm e- }04  &    &  8 & 2.791{\rm e- }01 \\
    & 16 & 7.684{\rm e- }04 &    & 16  & 1.658{\rm e- }04 &    & 16 & 1.811{\rm e- }04  &    & 16 & 3.565{\rm e- }03 \\
    & 32 & 7.684{\rm e- }04 &    & 32  & 1.658{\rm e- }04 &    & 32 & 1.811{\rm e- }04  &    & 32 & 1.126{\rm e- }03 \\
    & 64 & 7.684{\rm e- }04 &    & 64  & 1.658{\rm e- }04 &    & 64 & 1.811{\rm e- }04  &    & 64 & 1.126{\rm e- }03 \\
 \hline
 36 & 4  & 3.064{\rm e- }05 & 36 &  4  & 1.960{\rm e- }05 & 36 &  4 & 3.791{\rm e- }02  & 36 &  4 & 4.396{\rm e- }01 \\
    & 8  & 3.062{\rm e- }05 &    &  8  & 2.198{\rm e- }05 &    &  8 & 8.193{\rm e- }05  &    &  8 & 2.785{\rm e- }01 \\
    & 16 & 3.062{\rm e- }05 &    & 16  & 2.198{\rm e- }05 &    & 16 & 2.511{\rm e- }05  &    & 16 & 3.247{\rm e- }03 \\
    & 32 & 3.062{\rm e- }05 &    & 32  & 2.198{\rm e- }05 &    & 32 & 2.511{\rm e- }05  &    & 32 & 1.628{\rm e- }04 \\
    & 64 & 3.062{\rm e- }05 &    & 64  & 2.198{\rm e- }05 &    & 64 & 2.511{\rm e- }05  &    & 64 & 1.628{\rm e- }04 \\
 \hline
 54 & 4  & 5.528{\rm e- }07 & 54 &  4  & 1.679{\rm e- }05 & 54 &  4 & 3.790{\rm e- }02  & 54 &  4 & 4.395{\rm e- }01 \\
    & 8  & 5.581{\rm e- }07 &    &  8  & 3.130{\rm e- }07 &    &  8 & 7.193{\rm e- }05  &    &  8 & 2.784{\rm e- }01 \\
    & 16 & 5.581{\rm e- }07 &    & 16  & 3.130{\rm e- }07 &    & 16 & 2.889{\rm e- }07  &    & 16 & 3.163{\rm e- }03 \\
    & 32 & 5.581{\rm e- }07 &    & 32  & 3.130{\rm e- }07 &    & 32 & 2.889{\rm e- }07  &    & 32 & 9.813{\rm e- }06 \\
    & 64 & 5.581{\rm e- }07 &    & 64  & 3.130{\rm e- }07 &    & 64 & 2.889{\rm e- }07  &    & 64 & 9.813{\rm e- }06 \\
 \hline
\end{array}
\]
\caption{\label{tab:fpe}Estimated $L^2-$error for the solution of  fractional-order evolution equations at different time snapshots. }
\end{table}

\clearpage
\section*{Conclusions}
We have introduced and analyzed a quadrature method designed for integrals exhibiting both oscillatory behavior and exponential decay. Our approach employs Chebyshev-based interpolation to achieve efficient and accurate numerical integration. Additionally, we have developed an efficient algorithm for this rule whose computational cost remains largely independent of both the oscillatory behavior and exponential decay. Finally, we have established the method’s stability against round-off errors.

Future works may extend our {analysis} to broader classes of weight functions and investigate applications in high-dimensional quadrature problems

\paragraph{Acknowledgments}
{The author thanks Mahadevan Ganesh for the insightful discussions that motivated this work and for introducing him to the topic of evolution problems involving fractional derivatives. The author is also grateful to the referees of this manuscript for their careful reading, which helped to correct several errors, and for their suggestions, which simplified some arguments, improved the exposition, and enhanced the overall quality of the work.}

\paragraph{Funding} The author acknowledges the support of the projects ``Adquisición de cono\-ci\-mien\-to y
minería de datos, funciones especiales y métodos numéricos avanzados'' from Universidad Pública de Navarra, Spain, and ``Técnicas innovadoras para la resolución de problemas evolutivos'', ref. PID2022-136441NB-I00, from the Ministerio de Ciencia e Innovación, Gobierno de España.

\appendix
\end{document}